\documentclass[reqno,a4paper,12pt]{amsart} 

\usepackage{amsmath,amscd,amsfonts,amssymb}
\usepackage{mathrsfs,dsfont}
\usepackage{cases}

\numberwithin{equation}{section}
\numberwithin{figure}{section}

\newcommand\R{\mathbb{R}}
\newcommand\C{\mathbb{C}}

\newcommand\Z{\mathbb{Z}}

\newcommand\T{\mathbb{T}}
\newcommand\al{\alpha}
\newcommand\be{\beta}
\newcommand\gam{\gamma}

\newcommand\del{\delta}

\newcommand\lam{\lambda}
\newcommand\Lam{\Lambda}

\newcommand\sig{\sigma}
\newcommand\Om{\Omega}

\newcommand\1{\mathds{1}}
\newcommand\eps{\varepsilon}

\renewcommand\le{\leqslant}

\renewcommand\leq{\leqslant}
\renewcommand\geq{\geqslant}
\newcommand\sbt{\subset}

\newcommand{\mes}{\operatorname{mes}}

\newcommand{\supp}{\operatorname{supp}}

\newcommand\zal{\alpha \mathbb{Z} }
\newcommand\zbe{\beta \mathbb{Z} }
\newcommand\nz{n \mathbb{Z} }
\newcommand\mz{m \mathbb{Z} }
\newcommand\tal{\mathbb{T}_{\alpha}}
\newcommand\tbe{\mathbb{T}_{\beta}}
\newcommand\pal{\pi_{\alpha}}
\newcommand\pbe{\pi_{\beta}}

\newcommand\euler{\chi}

\newcommand{\interior}{\operatorname{int}}

\theoremstyle{plain}
\newtheorem{thm}{Theorem}[section]
\newtheorem{lem}[thm]{Lemma}
\newtheorem{cor}[thm]{Corollary}

\newtheorem{prop}[thm]{Proposition}

\newtheorem*{claim*}{Claim}

\newcommand{\thmref}[1]{Theorem~\ref{#1}}
\newcommand{\secref}[1]{Section~\ref{#1}}

\newcommand{\lemref}[1]{Lemma~\ref{#1}}
\newcommand{\defref}[1]{Definition~\ref{#1}}
\newcommand{\propref}[1]{Proposition~\ref{#1}}

\newcommand{\remref}[1]{Remark~\ref{#1}}

\theoremstyle{definition}
\newtheorem{definition}[thm]{Definition}
\newtheorem*{definition*}{Definition}
\newtheorem*{remarks*}{Remarks}
\newtheorem*{remark*}{Remark}
\newtheorem{remark}[thm]{Remark}

\newenvironment{enumerate-roman}
{\begin{enumerate}
\addtolength{\itemsep}{5pt}
}
{\end{enumerate}}

\newenvironment{enumerate-alph}
{\begin{enumerate}
\addtolength{\itemsep}{5pt}
}
{\end{enumerate}}

\newenvironment{enumerate-num}
{\begin{enumerate}
\addtolength{\itemsep}{5pt}
}
{\end{enumerate}}

\newenvironment{enumerate-text}
{\begin{enumerate}
\addtolength{\itemsep}{5pt}
}
{\end{enumerate}}

\begin{document}

\title{Functions tiling simultaneously with two arithmetic progressions}

\author{Mark Mordechai Etkind}
\address{Department of Mathematics, Bar-Ilan University, Ramat-Gan 5290002, Israel}
\email{mark.etkind@mail.huji.ac.il}

\author{Nir Lev}
\address{Department of Mathematics, Bar-Ilan University, Ramat-Gan 5290002, Israel}
\email{levnir@math.biu.ac.il}

\date{September 21, 2023}
\thanks{Research supported by ISF Grant No.\ 1044/21 and ERC Starting Grant No.\ 713927.}

\begin{abstract}
We consider measurable functions $f$ on $\mathbb{R}$ that tile simultaneously by two arithmetic progressions $\alpha \mathbb{Z}$ and $\beta \mathbb{Z}$ at respective tiling levels $p$ and $q$. We are interested in two main questions: what are the possible values of the tiling levels $p,q$, and what is the least possible measure of the support of $f$? We obtain sharp results which show that the answers depend on arithmetic properties of $\alpha, \beta$ and $p,q$, and in particular, on whether the numbers $\alpha, \beta$ are rationally independent or not.
\end{abstract}

\maketitle


\section{Introduction}

\subsection{} 
Let $f$ be a measurable function on $\R$, and 
$\Lam \sbt \R$ be a countable  set. 
We say that the function $f$
\emph{tiles $\R$
at level $w$} with the translation set $\Lam$,
or that 
\emph{$f+\Lam$ is a tiling of $\R$
at level $w$}
 (where $w$ is a constant), 
if we have
\begin{equation}
\label{eqI1.1}
\sum_{\lambda\in\Lambda}f(x-\lambda)=w\quad\text{a.e.}
\end{equation}
and the series in \eqref{eqI1.1} converges absolutely a.e.

In the same way one can define tiling 
by translates of a measurable function $f$ on 
$\R^d$, or  more generally, on any locally compact
 abelian group.

If $f = \1_\Omega$ 
is the indicator function of a set $\Omega$,
then $f + \Lam$ is a tiling at level one
if and only if 
the translated copies $\Omega+\lam$, $\lam\in\Lam$,
fill the whole space without overlaps up to measure zero. 
To the contrary, for tilings  by a
 general real or complex-valued function
$f$, the translated copies may have 
overlapping supports.

Tilings by translates of a function 
have been studied by various authors, see e.g.\ \cite{LM91}, 
\cite{KL96}, \cite{KW99}, \cite{Kol04}, \cite{KL16}, \cite{KW19},
\cite{Liu21}, \cite{KL21}, \cite{Lev22}, \cite{KP22}.

\subsection{}
By the \emph{support} of a function $f$ we shall mean the set
\begin{equation}
\label{eqI1.S}
\supp f := \{x : f(x) \neq 0\}.
\end{equation}

In \cite{KP22},
inspired by the Steinhaus tiling problem,
 the authors studied the following question: 
how ``small'' can be the 
support of a function $f$ which tiles 
$\R^d$ simultaneously by a finite number of lattices 
$\Lam_1, \dots, \Lam_N$?
In particular, they 
posed the question as to 
what is the least possible measure 
of the support of such a function $f$.

The problem is nontrivial even in 
dimension one and for two lattices only.
This case will be studied in the present paper.
We thus  consider  
a measurable function $f$ on $\R$
that simultaneously tiles 
by two arithmetic progressions
$\zal$ and $\zbe$, that is,
\begin{equation}
\label{eqA2.1.1}
\sum_{k \in \Z}f(x-k\al)=p,
\quad
\sum_{k \in \Z}f(x-k\be)=q \quad\text{a.e.}
\end{equation}
where $\al, \be$ are positive
real numbers, the tiling levels $p,q$ are complex numbers,
and both series in \eqref{eqA2.1.1} converge absolutely a.e.

It is obvious that if $p,q$ are both  nonzero, then
 the simultaneous tiling condition 
 \eqref{eqA2.1.1} implies that 
$\mes(\supp f)$
can be no smaller than $\max\{\al,\be\}$.
This estimate was improved 
for   \emph{nonnegative} functions 
$f$ in \cite[Theorem 2.6]{KP22},
where the authors proved  
that if $0 < \al < \be$ then
the tiling condition 
 \eqref{eqA2.1.1} implies that
$\mes(\supp f) \geq \lceil \be/\al \rceil \al$.
The authors asked in \cite[Question 4]{KP22} 
what is the least possible 
measure of the support of
 a function $f$ satisfying \eqref{eqA2.1.1}.
In this paper we obtain sharp results 
which improve on the lower bound from
\cite{KP22} and provide
a complete answer to this question.

\subsection{}
Notice that if $f$ is nonnegative, then 
  integrating the first equality in \eqref{eqA2.1.1}
over the interval $[0,\al)$ yields
$\int_{\R} f = p\al$, so $f$ must in fact be integrable.
The same holds if $f$ is complex valued
but assumed a priori to be in $L^1(\R)$.
Moreover, in this case we can also integrate
the second equality  in \eqref{eqA2.1.1}
over $[0,\beta)$ and get
 $\int_{\R}  f = q\be$, hence $p \al = q \be$. 
This proves the following basic fact:

\begin{prop}
\label{propA2.1}
Let $f$ be a measurable function on $\R$ assumed
to be either nonnegative or in $L^1(\R)$.
If $f$ satisfies \eqref{eqA2.1.1} then the vector
$(p,q)$ is proportional to $(\be, \al)$.
\end{prop}

The convolution $\1_{[0, \al)} \ast \1_{[0,\be)}$ provides a
basic  example
of a nonnegative function $f$ satisfying
 \eqref{eqA2.1.1} with $(p,q) = (\be, \al)$,
and such that $\supp f$ is an interval of length $\al + \be$.

We are interested in the following two main questions:

\begin{enumerate-roman}
\item 
Do there exist tilings \eqref{eqA2.1.1} such that the tiling
level vector $(p,q)$ 
\emph{is not} proportional to $(\be, \al)$? (In such a 
case $f$ can be neither  nonnegative nor integrable.)

\item 
What is the least possible value of $\mes(\supp f)$ 
for a function $f$ satisfying \eqref{eqA2.1.1}
with a given tiling level vector $(p,q)$?
\end{enumerate-roman}

In this paper we answer these questions in full generality.
The answers turn out to depend on arithmetic properties 
of $\alpha, \beta$ and $p,q$, and in particular,
on whether the numbers
$\al, \be$ are \emph{rationally  independent} or not. 
Moreover, we will see that  the results 
differ substantially between these two cases.


\section{Results}

\subsection{}
First we consider the case where
$\al, \be$ are rationally  independent.
In this case our first result establishes the existence of
tilings \eqref{eqA2.1.1} such that the levels $p,q$ are
 \emph{arbitrary complex numbers}, i.e.\ the vector
$(p,q)$ is not necessarily
proportional to $(\be, \al)$.
Moreover, we can construct
such tilings with $\mes(\supp f)$ never exceeding $\al + \be$.

\begin{thm}
\label{thmA3.1}
Let $\al, \be$ be rationally  independent.
For any two complex numbers $p,q$  
there is  a measurable function $f$
on $\R$ satisfying
\eqref{eqA2.1.1} with
$\mes(\supp f) \leq \al + \be$.
\end{thm}

We will also prove that while the function $f$ 
in \thmref{thmA3.1} has support of finite measure, 
$f$ cannot  in general 
be supported on any \emph{bounded} subset of $\R$.

 \begin{thm}
\label{thmS5.1}
Let $f$ be a measurable function on $\R$ satisfying \eqref{eqA2.1.1}
  where $\al, \be$ are rationally  independent. 
If  the vector $(p,q)$ is not proportional to $(\be, \al)$, then $\supp f$
 must be an unbounded set.
 \end{thm}

It is obvious that the result does not hold if $(p,q) = \lam (\be, \al)$
where $\lam$ is a scalar,
since in this case the   function $f = \lam \1_{[0, \al)} \ast \1_{[0,\be)}$
satisfies \eqref{eqA2.1.1} and has bounded support.

The next result  clarifies the role of the value $\al   + \be$
in \thmref{thmA3.1}. It turns out that for most level vectors
$(p,q)$  it is in fact the least possible value of $\mes(\supp f)$.

\begin{thm}
\label{thmA2.3}
Let $\al, \be$ be rationally  independent, and suppose
that $(p,q)$ is not proportional to any vector of
the form $(n,m)$ where $n,m$ are  nonnegative integers.
If a measurable function $f$ on $\R$ satisfies
\eqref{eqA2.1.1} then
$\mes (\supp f) \geq \al + \be$.
\end{thm}

In particular this result applies if $f$ 
is nonnegative, or is in $L^1(\R)$, or has bounded support.
It follows from \propref{propA2.1} and 
 \thmref{thmS5.1} that in any one of these cases  
the tiling level vector $(p,q)$ must be proportional to $(\be, \al)$, 
and since $\al, \be$ are rationally  independent,
$(p,q)$ cannot therefore be proportional to any
integer  vector $(n,m)$
unless $p,q$ are both zero. So
we obtain:

\begin{cor}
\label{corA2.2}
Assume that a measurable function $f$ on $\R$ is nonnegative, or is in $L^1(\R)$, or has bounded support.
If $\al, \be$ are rationally  independent and 
 \eqref{eqA2.1.1} holds for some
nonzero vector $(p,q)$, then 
$(p,q)$ is proportional to $(\be, \al)$ and
$\mes (\supp f) \geq \al + \be$.
\end{cor}

We thus obtain that for rationally  independent
$\al, \be$,  the convolution $\1_{[0, \al)} \ast \1_{[0,\be)}$ 
is a function
minimizing the value of $\mes(\supp f)$ among all
nonnegative, or all integrable, or all boundedly supported,
 functions $f$ satisfying
  \eqref{eqA2.1.1} for some nonzero tiling level vector $(p,q)$.

\subsection{}
We now consider the remaining case 
not covered by \thmref{thmA2.3}, namely,
the case where the tiling level vector
 $(p,q)$ is proportional to some vector
$(n,m)$ such that $n,m$ are nonnegative integers.
By multiplying the vector $(p,q)$ on an
appropriate scalar we may suppose that
$p,q$ are by themselves nonnegative integers,
and by factoring out their greatest common divisor
we may also assume $p,q$ to be \emph{coprime}.

Interestingly, it turns out that in this case
 the measure of $\supp f$ can drop below $\al + \be$,
in a magnitude that depends on the specific values
of the tiling levels $p$ and $q$.

\begin{thm}
\label{thmA2.5}
Let $\al, \be$ be rationally  independent, and let
 $p,q$ be two positive coprime integers. 
For any $\eps>0$ there is
a measurable function $f$
on $\R$ satisfying
\eqref{eqA2.1.1} such that
\begin{equation}
\label{eqA2.5.1}
\mes(\supp f) < \al + \be - \min \Big\{\frac{\al}{q}, \frac{\be}{p} \Big\} + \eps.
\end{equation}
\end{thm}

The next result shows
that the upper estimate \eqref{eqA2.5.1}
is actually sharp.

\begin{thm}
\label{thmA2.4}
Let $f$ be a measurable function on $\R$ satisfying
\eqref{eqA2.1.1} where 
$\al, \be$ are rationally  independent
and $p,q$ are positive, coprime integers. Then
\begin{equation}
\label{eqA2.4.1}
\mes(\supp f) > \al + \be - \min \Big\{\frac{\al}{q}, \frac{\be}{p} \Big\}.
\end{equation}
\end{thm}

The last two results yield that if
 the tiling levels $p,q$ are positive, coprime integers,
then  the right hand side of \eqref{eqA2.4.1}
is the infimum of the values of
$\mes(\supp f)$ over all measurable 
functions $f$ satisfying
\eqref{eqA2.1.1}, but this
infimum cannot be attained.

In Theorems \ref{thmA2.5} and \ref{thmA2.4}
the tiling levels $p,q$ are assumed to be both nonzero,
which does not cover the case where
$(p,q) = (1,0)$ or $(0,1)$. 
The following result provides the sharp answer in this last case.
By symmetry, it is enough to consider
 $(p,q)=(1,0)$.

\begin{thm}
\label{thmA2.6}
Let $\al, \be$ be rationally  independent,
and let $(p,q)=(1,0)$.
For any $\eps>0$ there is
a measurable function $f$
on $\R$ satisfying
\eqref{eqA2.1.1} 
such that $\mes(\supp f) < \al + \eps$.
Conversely,  any measurable $f$ satisfying
\eqref{eqA2.1.1} must have
$\mes(\supp f) > \al$.
\end{thm}

The results above thus fully resolve
the problem for rationally  independent $\al, \be$.

\subsection{}
We now move on to deal with the other case where
$\al, \be$ are linearly dependent over the rationals.
Then the vector $(\al,\be)$ is 
proportional to some vector
$(n,m)$ such that $n,m$ are positive integers.
By rescaling, it is enough to consider the
case $(\alpha,\beta) = (n,m)$ where $n,m$ are
positive integers. 

The tiling condition  \eqref{eqA2.1.1}  thus takes the form
\begin{equation}
\label{eqA2.7.1}
\sum_{k \in \Z}f(x-kn)=p,
\quad
\sum_{k \in \Z}f(x-km)=q \quad\text{a.e.}
\end{equation}
where $n,m$ are positive 
integers, $p,q$ are complex numbers,
and both series in \eqref{eqA2.7.1} converge absolutely a.e.

In this case our first result shows that 
the tiling levels $p,q$ cannot be arbitrary.

\begin{thm}
\label{thmA4.1}
Let $n,m$ be positive 
integers, and let
$f$ be a measurable function on $\R$ 
satisfying \eqref{eqA2.7.1}. Then the vector
$(p,q)$ must be proportional to $(m,n)$.
\end{thm}

This is not quite obvious since $f$ is  neither
assumed to be   
nonnegative nor in $L^1(\R)$, so the
conclusion does not follow from
\propref{propA2.1}.
Moreover, \thmref{thmA4.1} is in
sharp contrast to \thmref{thmA3.1}
which states that for rationally  independent
$\al, \be$ there exist tilings 
\eqref{eqA2.1.1}  such that the levels $p,q$ are
arbitrary complex numbers.

The next result gives a lower bound
for the support size of a function $f$
that satisfies the simultaneous tiling
condition \eqref{eqA2.7.1}
with a nonzero tiling level vector $(p,q)$.

\begin{thm}
\label{thmA4.3}
Let $f$ be a measurable function on $\R$ satisfying
\eqref{eqA2.7.1} where 
$n,m$ are positive integers
and the vector $(p,q)$ is nonzero. Then
\begin{equation}
\label{eqA4.3.11}
\mes(\supp f) \geq n + m - \gcd(n,m).
\end{equation}
\end{thm}

We will also establish that  in fact
the lower bound in \thmref{thmA4.3}
is  sharp.
Due to \thmref{thmA4.1}, it
 suffices
to prove this for the tiling level vector 
$(p,q) = (m,n)$.

\begin{thm}
\label{thmA4.2}
Let $n,m$ be positive 
integers, and let
$(p,q)=(m,n)$. Then there is a non\-negative,
measurable function $f$
on $\R$ satisfying \eqref{eqA2.7.1} and such that
$\supp f$ is an
interval of length $n+m-\gcd(n,m)$.
\end{thm}

It follows  that $n + m - \gcd(n,m)$ is 
the least possible
value of $\mes(\supp f)$
 among all 
 measurable  functions $f$ satisfying \eqref{eqA2.7.1}
with a nonzero tiling level vector $(p,q)$.
In particular, the convolution $\1_{[0, n)} \ast \1_{[0,m)}$
(whose support is an interval of length $n+m$)
\emph{does not} attain the least possible
value of $\mes(\supp f)$.

The results obtained thus answer the questions above
in full generality.

\begin{remark}
We note that the  case where the
tiling levels $p,q$ are both zero
is trivial, since then the zero function $f$
satisfies \eqref{eqA2.1.1}. It is also easy to
construct examples where $\supp f$ has positive but
arbitrarily small measure. For example, let $h$
be any function with $\supp h = (0, \eps)$,
then the function
$f(x) = h(x) - h(x + \al) - h(x + \be) + h(x + \al + \be)$
satisfies \eqref{eqA2.1.1} with $p,q$ both zero and $\supp f$ has positive
measure not exceeding $4 \eps$.
\end{remark}

\subsection{}
The rest of the paper  is organized as follows. 

In \secref{sect:prelim} we give a short
 preliminary background and fix notation that will be 
used throughout the paper.

In \secref{secE4} we prove Theorems \ref{thmA3.1},
 \ref{thmA2.5} and \ref{thmA2.6}, that is, 
for any two rationally  independent  $\al, \be$ 
and for any tiling level vector $(p,q)$,
we construct a simultaneous tiling  \eqref{eqA2.1.1} 
  such that $\mes(\supp f)$ is minimal, 
or is arbitrarily close to the infimum.

In \secref{secY1} we prove that if  a measurable function $f$ 
satisfies the simultaneous tiling condition 
\eqref{eqA2.1.1}  with a tiling level vector $(p,q)$ 
that is not proportional to $(\be, \al)$,
 then $\supp f$  must be an unbounded set
 (\thmref{thmS5.1}).

In \secref{secY2}  we  solve a problem posed to us by Kolountzakis,
asking whether  
there exists a \emph{bounded} measurable function $f$ on $\R$
that tiles simultaneously with 
rationally  independent  $\al, \be$ and with
arbitrary tiling levels $p,q$.
We  prove that  the 
answer is affirmative, and moreover,  $f$ 
can be chosen \emph{continuous and vanishing at infinity}.

In \secref{secE6} we prove Theorems 
\ref{thmA2.3} and \ref{thmA2.4}
 that give sharp lower bounds for
  the measure  of $\supp f$, where $f$ is 
 any measurable function 
satisfying  the simultaneous tiling condition 
\eqref{eqA2.1.1} 
with rationally  independent  $\al, \be$.

In the last \secref{secE9}, we consider the case where   
the two numbers $\al, \be$ are linearly dependent over the rationals.
By rescaling we may assume that 
$\al, \be$ are two positive integers 
$n,m$. We prove Theorems 
\ref{thmA4.1}, \ref{thmA4.3} and \ref{thmA4.2}
using a reduction of
the simultaneous  tiling problem
from the real line $\R$ to the 
set of integers $\Z$.


\section{Preliminaries. Notation.}
\label{sect:prelim}

In this section we give a short
 preliminary background and fix notation that will be 
used throughout the paper.

If $\al$  is a positive real number, then
we use $\tal$ to denote the circle group $\R / \zal$.
We let $\pal$ denote the canonical projection map
$\R \to \tal$.
The Lebesgue measure on the group $\tal$ is normalized
such that $\mes(\tal) = \al$.

We use $m(E)$, or $\mes(E)$,
to denote the Lebesgue measure of  a set $E$
in either the real line $\R$ or the circle $\tal$.

If
  $\al,\be$ are
two positive real numbers, then they are
  said to be \emph{rationally independent}
if the condition  $n \al + m \be = 0$, 
$n,m \in \Z$, implies that
$n=m=0$. 
This is the case if and only if the ratio $\al / \be$
is an irrational number. 

By the classical
Kronecker's theorem, 
if  two positive real numbers $\al,\be$ are rationally independent
then the sequence $\{\pal(n \be)\}$, $n=1,2,3,\dots$, is dense in $\tal$.

Let  $f$ be a measurable function on $\R$, 
and suppose that the series 
 \begin{equation}
\label{eqPR1.1}
\sum_{k \in \Z}f(x-k\al)
\end{equation}
converges absolutely for every $x \in \R$.
Then the sum \eqref{eqPR1.1}
is an $\al$-periodic function of $x$, so it
can be viewed as a function 
on $\tal$. We  denote this function  by $\pal(f)$.
If the sum \eqref{eqPR1.1} converges
absolutely not everywhere but almost everywhere,
then the function $\pal(f)$ is defined
in a similar way  on a full measure subset of $\tal$.

We observe that the simultaneous tiling condition
\eqref{eqA2.1.1} can be equivalently stated 
as the requirement that 
  $\pal(f) = p$ a.e.\ on $\tal$,
and that $\pbe(f) = q$ a.e.\ on $\tbe$.

If  $f$  is in $L^1(\R)$,
then the sum \eqref{eqPR1.1} converges
absolutely almost everywhere, and moreover,
the function 
$\pal(f)$ is in $L^1(\tal)$ and satisfies
  $\int_{\tal} \pal(f) = \int_{\R} f$.

The set
$\supp f := \{x  : f(x) \neq 0\}$
will be called   the \emph{support}  of the function
 $f$. If  we have $\supp f \sbt \Om$ then
we will say that \emph{$f$ is supported on $\Om$}.

We observe that
if $\supp f$ is a set of finite measure in $\R$, 
then in the sum \eqref{eqPR1.1}
 there are
 only finitely many nonzero terms
for almost every $x \in \R$, 
which implies that the function $\pal(f)$ 
is well defined on a full measure subset of $\tal$.


\section{Incommensurable arithmetic progressions:
Constructing simultaneously tiling functions  with small support}
\label{secE4}

In this section we prove Theorems \ref{thmA3.1},
 \ref{thmA2.5} and \ref{thmA2.6}, that is, 
for any two rationally  independent  $\al, \be$ 
and for any tiling level vector $(p,q)$,
we construct a simultaneous tiling  \eqref{eqA2.1.1} 
  such that $\mes(\supp f)$ is minimal, 
or is arbitrarily close to the infimum.

Throughout this section we 
shall suppose that $\al, \be > 0$ are two fixed,
rationally  independent real numbers.

\subsection{}
It will be convenient to introduce
the following terminology:

\begin{definition}
By an \emph{elementary  set}  
(in either $\R$, $\tal$ or $\tbe$) 
 we   mean a set which
can be represented as the union of  finitely many
 disjoint closed intervals of finite length.
\end{definition}

We will use $\interior(U)$ to denote the interior of an elementary set $U$.

\begin{lem}
  \label{lemFL1.1}
Let $A$ be an elementary set in $\tal$.
Then given any nonempty open interval  $J \sbt \tbe$,
no matter how small, 
one can find an elementary set $U \sbt \R$ 
such that
\begin{enumerate-roman}
\item
  \label{lemFL1.1.1}
$\pal(U) = A$;

\item
  \label{lemFL1.1.2}
$\pal$ is one-to-one on $\interior (U)$;

\item
  \label{lemFL1.1.3}
$\pbe(U) \sbt J$.
\end{enumerate-roman}
Moreover, $U$ can be chosen inside the
half-line $(r, +\infty)$ for any given number $r$.
\end{lem}

\begin{proof}
We choose $\delta>0$ smaller than
both the length of $J$ and $\al$,
and we decompose the elementary set $A$ as a union
$A = A_1 \cup \dots \cup A_n$,
where each $A_j$ is a closed interval in $\tal$
of length smaller than $\delta$,
and $A_1, \dots, A_n$ have disjoint
interiors. Let $U_j$ be a closed interval in $\R$
such that $A_j$ is a one-to-one image of $U_j$
under $\pal$.
By translating the sets $U_j$ by appropriate
integer multiples of $\al$ we can ensure that
$\pbe(U_j) \sbt J$ (due to 
Kronecker's theorem,  since
$\al, \be$ are rationally  independent),
 and that the sets
$U_1, \dots, U_n$ are pairwise disjoint and all
of them are contained
in a given half-line $(r, +\infty)$.  Then 
the set 
$U := U_1 \cup \dots \cup U_n$ 
is an elementary 
set contained in 
$(r, +\infty)$ and satisfying the properties 
  \ref{lemFL1.1.1}, \ref{lemFL1.1.2} and \ref{lemFL1.1.3}
above.
\end{proof}

\begin{lem}
  \label{lemFL4.1}
Let $A \sbt \tal$ be an elementary set,
and   $\varphi$ be a measurable function on $A$.
Given any nonempty open interval  $J \sbt \tbe$,
one can find an elementary set $U \sbt \R$ 
and a measurable function $f$ on $\R$,
such that
\begin{enumerate-roman}

\item
  \label{lemFL4.1.1}
$\pal(U) = A$;

\item
  \label{lemFL4.1.2}
$\pbe(U) \sbt J$;

\item
  \label{lemFL4.1.3}
$m(U) = m(A)$;

\item
  \label{lemFL4.1.4}
$f$ is supported on $U$;

\item
  \label{lemFL4.1.5}
$\pal(f) = \varphi$ a.e.\ on $A$.

\end{enumerate-roman}
Moreover, $U$ can be chosen inside the
half-line $(r, +\infty)$ for any given number $r$.
\end{lem}

\begin{proof}
Use \lemref{lemFL1.1} to
find an elementary set  $U \sbt \R$ such that
$\pal(U) = A$, 
$\pal$ is one-to-one on $\interior (U)$,
and 
$\pbe(U) \sbt J$.
Notice that the first two properties
imply that $m(U) =m(A)$. Recall also
that \lemref{lemFL1.1} allows us
to choose the set $U$ inside any given
half-line $(r, +\infty)$.
We define a function
$f$ on $\R$ by
$f(x) := \varphi(\pal(x))$ 
for $x \in \interior (U)$,
and $f(x)=0$ outside of 
$\interior (U)$. 
Then $f$ is a measurable function  
 supported on $U$. Since
$\pal$ is one-to-one on $\interior (U)$,
we have 
$\pal(f) = \varphi$ on the set $\pal(\interior(U))$,
a full measure subset of $A$.
The properties \ref{lemFL4.1.1}--\ref{lemFL4.1.5}
are thus satisfied and the claim is proved.
\end{proof}

\subsection{}
The next lemma incorporates a central idea of our
tiling construction technique. Roughly speaking, the
lemma  asserts that one can find a function $f$ on $\R$
with prescribed projections $\pal(f)$ and $\pbe(f)$, and
that,  moreover,
$\mes(\supp f)$ need never exceed  the total
measure of the supports of the projections.

\begin{lem}
  \label{lemFL1.2}
Suppose that we are given 
two elementary sets
$A \sbt \tal$, $B \sbt \tbe$, both of positive measure,
 as well as
two measurable functions $\varphi$ on $A$, and $\psi$ on $B$.
Then there is a closed 
set $\Om \sbt \R$ (a union of countably many
disjoint closed  intervals accumulating
at $+\infty$)
and a measurable function $f$ supported on $\Om$,
such that
\begin{enumerate-roman}
\item
  \label{lemFL1.2.3}
$m(\Om) = m(A) + m(B)$
(in particular,
$\Om$ has finite measure);

\item
  \label{lemFL1.2.1}
$\pal(\Om) = A$,  $\pal(f) = \varphi$ a.e.\ on $A$;

\item
  \label{lemFL1.2.2}
$\pbe(\Om) = B$,  $\pbe(f) = \psi$ a.e.\ on $B$.
\end{enumerate-roman}
Moreover, $\Om$ can be chosen inside the
half-line $(r, +\infty)$ for any given number $r$.
\end{lem}

\begin{proof}
We choose an arbitrary decomposition of the set $A$ as a union
$A = \bigcup_{k=1}^{\infty} A_k$,
where each $A_k \sbt \tal$ is an elementary set 
and the sets $A_1,A_2, \dots$ have nonempty and disjoint
interiors. We do the same also for the set $B$, that is,
we let 
$B = \bigcup_{k=1}^{\infty} B_k$,
where the $B_k$ are elementary sets in $\tbe$
with nonempty, disjoint
interiors. 

Now, we apply \lemref{lemFL4.1}
to the elementary set $A_1$, the function $\varphi$,
 and an arbitrary nonempty
open interval $J \sbt B$. We obtain from the lemma
an elementary set $U_1 \sbt \R$ and a measurable
 function $g_1$ on $\R$, satisfying the conditions
$\pal(U_1) = A_1$,
$\pbe(U_1) \sbt B$,
$m(U_1) = m(A_1)$,
the function  $g_1$ is supported on $U_1$,  and 
$\pal(g_1) = \varphi$ a.e.\ on $A_1$.

Next, we apply \lemref{lemFL4.1} again but
with the roles of
$\al,\be$ interchanged, 
to the elementary set $B_1$, the function $\psi - \pbe(g_1)$,
 and an arbitrary nonempty
open interval $J \sbt A \setminus A_1$. The lemma yields
an elementary set $V_1 \sbt \R$  and a measurable
 function $h_1$ on $\R$,
such that
$\pbe(V_1) = B_1$, 
 $\pal(V_1) \sbt A \setminus A_1$,
$m(V_1) = m(B_1)$,
the function  $h_1$ is supported on $V_1$,  and 
$\pbe(h_1) = \psi - \pbe(g_1)$ a.e.\ on $B_1$.

We continue the construction in a similar fashion. 
Suppose that we have already constructed the 
sets $U_k, V_k$ and the functions $g_k, h_k$ for
 $1 \leq k \leq n-1$.
Using \lemref{lemFL4.1} we find an elementary 
set $U_n \sbt \R$  and a measurable
 function $g_n$ on $\R$, such that
$\pal(U_n) = A_n$,
 $\pbe(U_n) \sbt B \setminus \bigcup_{k=1}^{n-1} B_k$,
$m(U_n) = m(A_n)$,
  $g_n$ is supported on $U_n$,  and 
 \begin{equation}
\label{eqFL2.1}
\pal(g_n) = \varphi - \sum_{k=1}^{n-1} \pal(h_k)
\quad \text{a.e.\ on $A_n$.}
\end{equation}
Then, we use again 
\lemref{lemFL4.1} to find an elementary 
set $V_n \sbt \R$   and a measurable
 function $h_n$ on $\R$, such that
$\pbe(V_n) = B_n$, 
$\pal(V_n) \sbt A \setminus \bigcup_{k=1}^{n} A_k$,
$m(V_n) = m(B_n)$,
  $h_n$ is supported on $V_n$,  and 
 \begin{equation}
\label{eqFL2.2}
\pbe(h_n) =  \psi - \sum_{k=1}^{n} \pbe(g_k)
\quad \text{a.e.\ on $B_n$.}
\end{equation}

We may assume that the  
sets $U_1, V_1, U_2, V_2, \dots$  
are pairwise disjoint, 
that all of them lie  inside a given
half-line $(r, +\infty)$, and that they 
accumulate at $+ \infty$. Indeed,  \lemref{lemFL4.1} 
allows us to choose the sets 
such that they satisfy these properties.
(In fact, one can check  that by their
construction the sets necessarily have
 disjoint interiors.)

 Finally, we define
 \begin{equation}
\label{eqFL2.5}
\Om := \bigcup_{n=1}^{\infty} (U_n \cup V_n),
\quad
f  := \sum_{n=1}^{\infty} (g_n + h_n).
\end{equation}
The sum on the right hand side of \eqref{eqFL2.5}
is well defined as the
terms in the series have disjoint supports.
 We will show that the properties
\ref{lemFL1.2.3}, \ref{lemFL1.2.1}
and \ref{lemFL1.2.2}
are satisfied.

We begin by verifying that
\ref{lemFL1.2.3} holds. Indeed, 
$m(U_n) = m(A_n)$, $m(V_n) = m(B_n)$,
and the sets 
$U_1, V_1, U_2, V_2, \dots$  
are  disjoint. It follows that
 \begin{equation}
\label{eqFL2.19}
m(\Om) = \sum_{n=1}^{\infty} (m(U_n) + m(V_n))
 = \sum_{n=1}^{\infty} (m(A_n) + m(B_n))
 = m(A) + m(B).
\end{equation}

Next we verify that
 \ref{lemFL1.2.1} is satisfied. Indeed,
we have $\pal(U_n) = A_n$ and $\pal(V_n) \sbt A$
for every $n$,  hence $\pal(\Om) = A$.
We must show that also
$\pal(f) = \varphi$ a.e.\ on $A$. It would suffice to verify
that this holds  on each $A_n$. Notice that
$\pal(\interior (U_k))$ is disjoint from $A_n$
for $k \neq n$, and $\pal(V_k)$ 
 is disjoint from $A_n$ for $k \geq n$.
Hence using \eqref{eqFL2.1} this implies that
 \begin{equation}
\label{eqFL2.20}
\pal(f) = \pal(g_n) + \sum_{k=1}^{n-1} \pal(h_k) = \varphi
\quad \text{a.e.\ on $A_n$.}
\end{equation}

In a similar way we can
show that \ref{lemFL1.2.2} holds as well. 
We have $\pbe(V_n) = B_n$ and $\pbe(U_n) \sbt B$
for every $n$, hence $\pbe(\Om) = B$.
To see that 
$\pbe(f) = \psi$ a.e.\ on $B$, we verify
that this is the case  on each $B_n$. But 
$\pbe(\interior (V_k))$ is disjoint from $B_n$
for $k \neq n$, and $\pbe(U_k)$ 
 is disjoint from $B_n$ for $k \geq n+1$.
Hence  \eqref{eqFL2.2} implies that
 \begin{equation}
\label{eqFL2.21}
\pbe(f) = \pbe(h_n) + \sum_{k=1}^{n} \pbe(g_k) = \psi
\quad \text{a.e.\ on $B_n$.}
\end{equation}
Thus all the properties
\ref{lemFL1.2.3}, \ref{lemFL1.2.1}
and \ref{lemFL1.2.2}
are satisfied and
 \lemref{lemFL1.2} is proved.
\end{proof}

\subsection{}
We can now use 
 \lemref{lemFL1.2}
in order to prove 
\thmref{thmA3.1} 
and \thmref{thmA2.6}.

\begin{proof}[Proof of \thmref{thmA3.1}]
Let $p,q$ be 
any two complex numbers.
Apply \lemref{lemFL1.2}
to the sets $A = \tal$, $B=\tbe$, and
to the constant functions
$\varphi = p$, $\psi = q$. The  lemma yields
a measurable function $f$ on $\R$,
 supported on a set
$\Om \sbt \R$  of  measure $\al + \be$,
and such that 
$\pal(f) = p$ a.e.\ on $\tal$,
while
$\pbe(f) = q$ a.e.\ on $\tbe$,
that is, $f$
satisfies the tiling condition
\eqref{eqA2.1.1}.
The theorem is thus proved.
\end{proof}

\begin{proof}[Proof of \thmref{thmA2.6}]
Let $(p,q)=(1,0)$. Given  $\eps>0$ we
   apply \lemref{lemFL1.2}
to the sets $A = \tal$ and $B=[0, \eps] \sbt \tbe$, and
to the functions
$\varphi = 1$, $\psi = 0$. The lemma yields
a set $\Om \sbt \R$ satisfying
$m(\Om) = \al + \eps$, $\pbe(\Om) = B$,
as well as a measurable function $f$ supported on $\Om$
and such that
$\pal(f) = 1$ a.e.\ on $\tal$,
and $\pbe(f) = 0$ a.e.\ on $B$. Notice though
that the condition
$\pbe(\Om) = B$ ensures that 
$\pbe(f) = 0$ a.e.\ also on $\tbe \setminus B$.
Hence the tiling condition
\eqref{eqA2.1.1} is satisfied.
This proves one part of the theorem.

To prove the converse part, we suppose that
$f$ is a measurable function on $\R$ satisfying
\eqref{eqA2.1.1} 
with $(p,q)=(1,0)$, that is,  
$\pal(f) = 1$ a.e.\ on $\tal$
and $\pbe(f) = 0$ a.e.\ on $\tbe$.
 It follows from the first assumption 
that the set $\Om  := \supp f$ has
measure at least $\al$, since
 $\pal(\Om)$ is a set of full
measure in $\tal$. 
We  must show that actually
$m(\Om) > \al$.
Suppose to the contrary that $m(\Om) = \al$.
Then $\pal(\Om)$ cannot be a set of full
measure in $\tal$ unless $\pal$ is one-to-one
on a full measure subset of $\Om$. 
But then the assumption that
$\pal(f) = 1$ a.e.\ on $\tal$
implies that $f = 1$ a.e.\ on its support $\Om$, 
which in turn contradicts  our 
second assumption that
$\pbe(f) = 0$ a.e.\ on $\tbe$.
Hence we must have
$m(\Om) > \al$, and so 
the second part of the theorem is also proved. 
\end{proof}

\subsection{}
\label{subsecDB1}
Next we turn to prove 
\thmref{thmA2.5}.
The proof will require the following notion:

\begin{definition}
\label{defDBA}
 An $n \times m$ matrix $M = (c_{ij})$ is called
a \emph{doubly stochastic array} 
(with uniform marginals)   if
the entries $c_{ij}$ are nonnegative, and 
\begin{equation}
\label{eqDS1.1}
\sum_{j=1}^{m} c_{ij} = m, \quad 1 \leq i \leq n,
\end{equation}
\begin{equation}
\label{eqDS1.2}
\sum_{i=1}^{n} c_{ij} = n, \quad 1 \leq j \leq m,
\end{equation}
that is, the sum of the entries at each row is $m$
and at each   column  is $n$.
\end{definition}

By the \emph{support} of the
matrix $M = (c_{ij})$ we refer to the set
\[
\supp M = \{(i,j) : c_{ij} \neq 0\}.
\]

In \cite[Question 7]{KP22} the authors posed the following question,
which arose in connection with the simultaneous tiling problem
in  finite abelian groups: what is
the least possible size of the support of a
  doubly stochastic  $n \times m$ array?

This problem was solved recently in  \cite{Lou23} and 
independently  in \cite{EL22}.

\begin{thm}[{\cite{Lou23}, \cite{EL22}}]
\label{thmC1.6}
For all $n,m$ the minimal size of the
support of an $n \times m$
doubly stochastic array is
equal to $n + m - \gcd(n,m)$.
\end{thm}

In particular, there exists an $n \times m$
doubly stochastic array whose support size 
is as small as  $n + m - \gcd(n,m)$.
We will use this  fact in the  proof 
of  \lemref{lemFN1.2} below.

\subsection{}

\begin{lem}
  \label{lemFN1.1}
Let  $p,q$ be two positive integers, and
$0 < \gam <  \min \{\al q^{-1}, \be p^{-1} \}$.
Then there is a system 
$\{L_{ij}\}$,
 $ 1 \leq i \leq p$, $ 1 \leq j \leq q$,
of disjoint closed intervals in $\R$, 
with the following properties:
\begin{enumerate-roman}
\item
  \label{lemFN1.1.1}
each interval $L_{ij}$ has length $\gam$;

\item
  \label{lemFN1.1.2}
$\pal(L_{ij})$ is an interval $I_j \sbt \tal$ that does  not  depend on $i$;

\item
  \label{lemFN1.1.3}
$\pbe(L_{ij})$ is an interval $J_i  \sbt \tbe$  that does  not  depend on $j$;

\item
  \label{lemFN1.1.4}
$I_1, \dots, I_q$ are disjoint closed  intervals in $\tal$;

\item
  \label{lemFN1.1.5}
$J_1, \dots, J_p$ are disjoint closed intervals in $\tbe$.

\end{enumerate-roman}
\end{lem}

\begin{proof}
We choose integers $m_1, \dots, m_q$ such that
the intervals  $I_j := [m_j \be,  m_j \be + \gam]$,
$1 \leq j \leq q$, 
are disjoint in $\tal$. This is possible due to 
Kronecker's theorem,  since
$\al, \be$ are rationally  independent and  $q \gam < \al$.
Since we also have
 $p \gam < \be$, we can find
in a similar way integers $n_1, \dots, n_p$ such that
the intervals  $J_i := [n_i \al,  n_i \al + \gam]$,
$1 \leq i \leq p$, are disjoint in $\tbe$.
We then define the intervals  $L_{ij} \sbt \R$ by
 \begin{equation}
\label{eqFL3.1}
L_{ij} := [0,\gam] + n_i \al + m_j \be,
\quad  1 \leq i \leq p, \; 1 \leq j \leq q.
\end{equation}
Then each interval $L_{ij}$ has
length $\gam$, and we have
$\pal(L_{ij}) = I_j$, $\pbe(L_{ij}) = J_i$,
so all the required properties
\ref{lemFN1.1.1}--\ref{lemFN1.1.5}
are satisfied.

Lastly we show that the intervals $L_{ij}$ must be
disjoint. Indeed, suppose that two intervals
$L_{i_1,j_1}$ and $L_{i_2,j_2}$ 
have a point $x$ in common.
Then, on one hand,  from
 \ref{lemFN1.1.2} we obtain
$\pal(x) \in I_{j_1} \cap I_{j_2}$, 
which in turn using \ref{lemFN1.1.4}
implies that $j_1 = j_2$.
On the other hand, 
by  \ref{lemFN1.1.3} we also have
$\pbe(x) \in J_{i_1} \cap J_{i_2}$, 
and hence $i_1 = i_2$   which now follows from
\ref{lemFN1.1.5}. Hence the intervals 
$L_{i_1,j_1}$ and $L_{i_2,j_2}$ 
cannot intersect unless  $(i_1, j_1) = (i_2, j_2)$.
\end{proof}

\begin{lem}
  \label{lemFN1.2}
Let  $p,q$ be two positive coprime integers, and 
  $0 < \gam <  \min \{\al q^{-1}, \be p^{-1} \}$.
Then there is an elementary set
$\Om \sbt \R$ and
a  measurable function $f$ supported on $\Om$,
such that
\begin{enumerate-roman}
\item
  \label{lemFN1.2.0}
$m(\Om) =  (p+q-1) \gam$;

\item
  \label{lemFN1.2.1}
the set $A = \pal(\Om)$ in $\tal$ has measure $q\gam$;

\item
  \label{lemFN1.2.2}
the set $B = \pbe(\Om)$ in $\tbe$ has measure $p\gam$;

\item
  \label{lemFN1.2.3}
$\pal(f) = p$ a.e.\ on $A$;

\item
  \label{lemFN1.2.4}
$\pbe(f) = q$ a.e.\ on $B$.
\end{enumerate-roman}
\end{lem}

It is instructive to compare this result with
\lemref{lemFL1.2} above. Recall  that 
the sets $A, B$ in Lemma \ref{lemFL1.2} 
can be any two elementary subsets of
$\tal$ and $\tbe$ respectively, and that the projections
$\pal(f)$, $\pbe(f)$  can be any two
 measurable functions on $A$ and $B$
respectively, but the measure of the support
$\Om$ must in general
be as large as  the sum of $m(A)$ and $m(B)$.
To the contrary, in \lemref{lemFN1.2} 
we are able to reduce the measure of 
the support $\Om$
to be strictly smaller than the sum of $m(A)$ and $m(B)$.

\begin{proof}[Proof of  \lemref{lemFN1.2}]
Let $\{L_{ij}\}$,
 $ 1 \leq i \leq p$, $ 1 \leq j \leq q$,
be the system of disjoint closed intervals  
given by \lemref{lemFN1.1}.
We use \thmref{thmC1.6}  to find a
$p \times q$
doubly stochastic array $M = (c_{ij})$,
whose support is of  size $p+q-1$
(which is the smallest possible size as $p,q$ are coprime).
We define a function $f$ on $\R$ by 
$f(x) := c_{ij}$ for $x \in L_{ij}$, 
 $ 1 \leq i \leq p$, $ 1 \leq j \leq q$,
and  $f(x):=0$ if $x$ does not
lie in any one of the intervals $L_{ij}$.

Let $\Om$ be  the support of the function $f$.
Then $\Om$ is the union
of those intervals  $L_{ij}$ for which $ (i,j)\in \supp M$.
Since $|\supp M| = p+q-1$, and since the intervals $L_{ij}$ are disjoint
and have length $\gam$, it follows that 
 $m(\Om) =  (p+q-1) \gam$.
 
Recall that  
$\pal(L_{ij})$ is a closed interval $I_j \sbt \tal$ 
not  depending on $i$,
and  the intervals 
$I_1, \dots, I_q$  are disjoint. Let 
$A = I_1 \cup \dots \cup I_q$,
then  $A$  has measure $q\gam$.
We show that
$\pal(f) = p$ a.e.\ on $A$. It would suffice to verify
that this holds on each one of the intervals $I_j$. Indeed,
as  the sum of the entries of the matrix
$M$ at the $j$'th column  is $p$, we have
 \begin{equation}
\label{eqFN1.3.1}
\pal(f) =   \sum_{i=1}^{p} c_{ij} = p \quad \text{on $I_j$.}
 \end{equation}

Next, in a similar way, we let $B = J_1 \cup \dots \cup J_p$,
then   $B$  has measure $p\gam$. We show that 
$\pbe(f) = q$ a.e.\ on $B$, by checking
that this holds on each $J_i$. And indeed, 
this time due to
the sum of the entries of 
$M$ at the $i$'th row being $q$, we get
 \begin{equation}
\label{eqFN1.3.2}
\pbe(f) =   \sum_{j=1}^{q} c_{ij} = q \quad \text{on $J_i$.}
 \end{equation}

Finally, since $p$ is nonzero,
it follows from \eqref{eqFN1.3.1}
 that $\pal(f)$ does not vanish on $A$, 
hence $\pal(\Om)$ must cover $A$. But $\pal(\Om)$ is a subset
of $A$, so we get $\pal(\Om) = A$.
Similarly, also $q$ is nonzero,
so \eqref{eqFN1.3.2} implies that
$\pbe(\Om)$ covers $B$, and since
$\pbe(\Om)$ is also a subset of $B$
we conclude that $\pbe(\Om) = B$.
The lemma is thus proved.
\end{proof}

\subsection{}
Now we are able to prove \thmref{thmA2.5}
based on the results above.

\begin{proof}[Proof of \thmref{thmA2.5}]
Let  $p,q$ be two positive coprime integers, and denote
\begin{equation}
\label{eqS1.2.1}
\sigma := \min \Big\{\frac{\al}{q}, \frac{\be}{p} \Big\}.
\end{equation}
 Given $\eps>0$,
we choose a number $\gam$ such that
$\sig - \eps < \gam < \sig$
(we can assume that $\eps$ is smaller than $\sig$).
We then use \lemref{lemFN1.2} to  obtain
an elementary set
$\Om_1 \sbt \R$ of measure $(p+q-1) \gam$,
and a  measurable function $f_1$ supported on $\Om_1$,
such that
the elementary set $A_1 := \pal(\Om_1)$  has measure $q\gam$,
the elementary set $B_1 := \pbe(\Om_1)$ has measure $p\gam$,
and such that $\pal(f_1) = p$ a.e.\ on $A_1$,
while $\pbe(f_1) = q$ a.e.\ on $B_1$.

Next, we apply \lemref{lemFL1.2}
to the two elementary sets 
$A_2 := \tal \setminus \interior(A_1)$, 
 $B_2 :=\tbe \setminus \interior(B_1)$,
 and
to the constant functions
$\varphi = p$, $\psi = q$. 
The lemma allows us to find 
a  set $\Om_2 \sbt \R$
and a measurable function $f_2$ supported on $\Om_2$,
such that
$\pal(\Om_2) = A_2$,  $\pal(f_2) = p$ a.e.\ on $A_2$,
$\pbe(\Om_2) = B_2$,  $\pbe(f_2) = q$ a.e.\ on $B_2$,
and
\begin{equation}
\label{eqSL2.16}
m(\Om_2) = m(A_2) + m(B_2) = (\al - q\gam) + (\be - p\gam) = 
\al + \be - (p+q)\gam.
\end{equation}
The lemma also allows us to choose $\Om_2$ to be
disjoint from $\Om_1$.

We now define
 \begin{equation}
\label{eqSL2.5}
\Om := \Om_1 \cup \Om_2,
\quad
f  := f_1 + f_2.
\end{equation}
Then $f$ is supported by $\Om$. Since
$\Om_1$ and $\Om_2$ are disjoint, we have
 \begin{equation}
\label{eqSL2.11}
m(\Om)   = m(\Om_1) + m(\Om_2) = \al + \be - \gam.
\end{equation}
But recall that we have chosen $\gam$ such that $\gam > \sig - \eps$,
hence \eqref{eqSL2.11} yields that
$\mes(\supp f) < \al + \be - \sig + \eps$. That
is, the condition \eqref{eqA2.5.1} is satisfied.

 We must verify that $f$ satisfies the tiling condition
 \eqref{eqA2.1.1}. We first show that
  $\pal(f) = p$ a.e.\ on $\tal$.
Indeed, we have
$\pal(\Om_1) = A_1$,
$\pal(\Om_2) = A_2$,
where $A_1, A_2$ have disjoint interiors
and their union is the whole $\tal$. 
Moreover, 
$\pal(f) = \pal(f_1) = p$ a.e.\ on $A_1$, and
$\pal(f) = \pal(f_2) = p$ a.e.\ on $A_2$,
which proves the claim.

Finally we show that also
  $\pbe(f) = q$ a.e.\ on $\tbe$.
In a similar way, we have  
$\pbe(\Om_1) = B_1$,
$\pbe(\Om_2) = B_2$,
and $B_1, B_2$ have disjoint interiors
and their union is $\tbe$. 
As before, we have
$\pbe(f) = \pbe(f_1) = q$ a.e.\ on $B_1$, and
$\pbe(f) = \pbe(f_2) = q$ a.e.\ on $B_2$.
This shows that the tiling condition
 \eqref{eqA2.1.1} indeed holds 
and thus the theorem is proved.
\end{proof}

\subsection{Remarks}
1. Let us say that a measurable function
$f$ on $\R$ is \emph{piecewise constant}
if there is a strictly increasing real sequence $\{\lam_n\}$,
$n \in \Z$, with no finite accumulation points, such that
$f$ is constant a.e.\ on each one of the 
intervals $[\lam_n, \lam_{n+1})$
(note that these intervals constitute a partition of $\R$).
One can verify that
our proof of  Theorems
\ref{thmA3.1}, \ref{thmA2.5} and
\ref{thmA2.6} yields a  function
$f$  which is not only
measurable, but in fact is
 piecewise constant on $\R$.

2. Our construction method allows us to choose the function $f$
in Theorems \ref{thmA3.1}, \ref{thmA2.5} and
\ref{thmA2.6} to have ``dispersed'' support,
that is, $f$ can be supported on the union of 
(countably many) small
intervals that are located far apart from each other.


\section{Simultaneous tiling by functions of bounded support}
\label{secY1}

 \subsection{}
One can easily notice that our proof of Theorems \ref{thmA3.1},
 \ref{thmA2.5} and \ref{thmA2.6} above
   yields a function $f$ 
 whose support lies inside any given half-line $(r, +\infty)$,  
 so that $\supp f$ is bounded from below. 
 One may ask whether the function $f$ can be chosen such that
  the support is bounded from both above 
and below at the same time. 

The answer is obviously `yes' if we have 
$(p,q) = \lam (\be, \al)$ where
$\lam$ is a scalar, since in this case the 
 function $f = \lam \1_{[0, \al)} \ast \1_{[0,\be)}$
 satisfies the simultaneous tiling condition  \eqref{eqA2.1.1} and has bounded support.

To the contrary, if the tiling level vector
$(p,q)$ is not proportional to $(\be, \al)$,
then \thmref{thmS5.1} provides the question above
with a negative answer: \emph{$f$ cannot be supported on any bounded set}. 
This theorem will be proved in the present  section.

We note that our proof in fact does not use the assumption that 
 $\al, \be$ are rationally  independent. 
 However if $\al, \be$ are linearly dependent over the rationals,
then we know from \thmref{thmA4.1} that there do not exist any
simultaneous tilings   \eqref{eqA2.1.1} with a   level
vector $(p,q)$ that is not proportional to $(\be, \al)$, 
so  the result is vacuous in this case.

 \subsection{}
 We now turn to prove \thmref{thmS5.1}. 
 To this end, we shall use  a result due to Anosov \cite[Theorem 1]{Ano73} 
that we state here as a lemma:

 \begin{lem}[{\cite{Ano73}}]
\label{lemT5.2}
Let $\varphi \in L^1(\tal)$. If the equation
 \begin{equation}
\label{eqT5.3}
\psi(x) - \psi(x + \be) = \varphi(x) \quad \text{a.e.}
\end{equation}
has a measurable solution $\psi: \tal \to \C$,  then $\int_{\tal} \varphi = 0$.
 \end{lem}

In fact, in \cite[Theorem 1]{Ano73} a more general version of this result 
was stated and proved, in the context of a measure-preserving
transformation acting on a finite measure space. Here we only state the 
result in the special case where the transformation is  a rotation of the circle $\tal$.

\begin{proof}[Proof of \thmref{thmS5.1}]
Assume that $f$ is a measurable function on $\R$ satisfying   \eqref{eqA2.1.1}.
We shall suppose that   $f$ has bounded support, and   prove
that this implies that the vector
$(p,q)$ must be proportional to $(\be, \al)$.

By translating $f$ we may assume that   $\supp f \sbt [0, n \be)$,
where $n$ is a positive, large enough integer. We can then write 
 \begin{equation}
\label{eqT5.5}
f = \sum_{k=0}^{n-1} f_k, \quad f_k  := f \cdot \1_{[k \be, (k+1)\be)}.
\end{equation}
By the first condition in \eqref{eqA2.1.1} we have
 \begin{equation}
\label{eqT5.9}
\sum_{k=0}^{n-1} \pal(f_k) = \pal(f) = p \quad \text{a.e.}
\end{equation}
The second condition in \eqref{eqA2.1.1}  can be equivalently stated as  
 \begin{equation}
\label{eqT5.6}
\sum_{k=0}^{n-1} f_k(x + k \be) = q \cdot \1_{[0, \be)}(x) \quad \text{a.e.}
\end{equation}
It follows from \eqref{eqT5.6} that
 \begin{equation}
\label{eqT5.7}
\sum_{k=0}^{n-1} \pal(f_k) (x + k \be) = q \cdot \pal(\1_{[0, \be)})(x) \quad \text{a.e.}
\end{equation}
Let us define
 \begin{equation}
\label{eqT5.12}
\varphi :=  p - q \cdot \pal(\1_{[0, \be)}),
 \quad \psi_k :=  \pal(f_k), \quad 0 \leq k \leq n-1,
\end{equation}
then   $\varphi \in L^1(\tal)$, while
the $\psi_k$ are measurable functions on $\tal$.
 If we now subtract the equality
 \eqref{eqT5.7} from  \eqref{eqT5.9}, this yields
 \begin{equation}
\label{eqT5.10}
\sum_{k=1}^{n-1}  ( \psi_k(x) - \psi_k (x + k \be) ) 
= \varphi(x) \quad \text{a.e.}
\end{equation}
Lastly, we introduce 
 a measurable function $\psi$ on $\tal$  defined by
 \begin{equation}
\label{eqT5.14}
\psi(x) := \sum_{k=1}^{n-1} \sum_{j=0}^{k-1} \psi_k(x + j \be),
\end{equation}
and observe that \eqref{eqT5.10} can be reformulated as
 \begin{equation}
\label{eqT5.13}
\psi(x) - \psi(x + \be) = \varphi(x) \quad \text{a.e.}
\end{equation}
This allows us to apply \lemref{lemT5.2}, which
yields  $\int_{\tal} \varphi = 0$. 
But using \eqref{eqT5.12} we have
 \begin{equation}
\label{eqT5.16}
\int_{\tal} \varphi = 
\int_{\tal} ( p - q \cdot \pal(\1_{[0, \be)}) )
= p \al - q \int_{\R} \1_{[0, \be)} = p \al - q \be.
\end{equation}
We conclude that  $p \al - q \be = 0$,  so the vector
$(p,q)$ is proportional to $(\be, \al)$.
\end{proof}


\section{Simultaneous tiling by  a bounded function}
 \label{secY2}

\subsection{}
The following question was posed to us
by Kolountzakis: Let
 $\al, \be$ be rationally  independent.
Given two
arbitrary complex numbers $p,q$, does  
there exist  a \emph{bounded} measurable function $f$
on $\R$ satisfying the simultaneous tiling condition
\eqref{eqA2.1.1}?

The answer is once again `yes' if we have 
$(p,q) = \lam (\be, \al)$, $\lam \in \C$, 
since in this case the continuous,
compactly supported  function
$f = \lam \1_{[0, \al)} \ast \1_{[0,\be)}$
satisfies \eqref{eqA2.1.1}.

On the other hand, 
the problem is nontrivial
if the vector
$(p,q)$ is not proportional to $(\be, \al)$. 
We note that in this case, a bounded function
$f$ satisfying \eqref{eqA2.1.1}
\emph{cannot be
supported on any set of finite  measure}.
Indeed,  if $\mes( \supp f)$ is finite
then $f$ must be in $L^1(\R)$, which is not
possible due to  \propref{propA2.1}.

We will nevertheless 
prove that  the question above 
admits an affirmative 
answer. Moreover, 
one can always choose $f$ to be
continuous and vanishing at infinity:

 \begin{thm}
\label{thmR25.1}
Let $\al, \be$ be rationally  independent.
For any two complex numbers $p,q$  
one can find a continuous function $f$
on $\R$ vanishing at infinity, and
satisfying \eqref{eqA2.1.1}.
 \end{thm}

\subsection{}
In what follows we assume $\al, \be$  to be rationally  independent. Our proof of 
\thmref{thmR25.1} is based on the 
technique used to prove 
\lemref{lemFL1.2},
but this time we will use the following variant of
\lemref{lemFL4.1}.

\begin{lem}
  \label{lemFC1.1}
Let $A$ be an elementary set in $\tal$, and
$\varphi$ be a continuous function on $A$.
Then given any $\del>0$ and any
nonempty open interval  $J \sbt \tbe$,
one can find an elementary set $U \sbt \R$ 
and a continuous function $f$ on $\R$
such that
\begin{enumerate-roman}
\item
  \label{lemFC1.1.1}
$\pal(U) = A$;

\item
  \label{lemFC1.1.2}
$\pbe(U) \sbt J$;

\item
  \label{lemFC1.1.3}
$f$ is supported on $U$,  $|f(x)| \leq \del$ for all $x \in U$;

\item
  \label{lemFC1.1.4}
$\pal(f) = \varphi$ on some elementary set $A' \sbt A$,
$m(A \setminus A') < \del$.
\end{enumerate-roman}
Moreover, $U$ can be chosen inside the
half-line $(r, +\infty)$ for any given number $r$.
\end{lem}

\begin{proof}
We apply \lemref{lemFL1.1}
to the elementary set $A$ and to the 
open interval $J$. The lemma yields
an elementary set $U_0 \sbt \R$ 
such that
$\pal(U_0) = A$, 
$\pal$ is one-to-one on $\interior (U_0)$,
and $\pbe(U_0) \sbt J$. 
Let $M := \sup |\varphi(t)|$, $t \in A$, and 
 choose an integer 
$N$ sufficiently large so that $N\delta > M$. We then
 find integers $m_1, \dots, m_N$ such that,
if we denote $U_j := U_0 + m_j \al$, 
$1 \leq j \leq N$,
then 
$\pbe(U_j) \sbt J$ for every $j$.
This is possible due to 
Kronecker's theorem,  since
$\al, \be$ are rationally  independent and 
$\pbe(U_0)$ is a compact subset of the
open interval $J$. We can also
choose the integers $m_1, \dots, m_N$ such that
the sets $U_1, \dots, U_N$ are disjoint and all
of them are contained
in a given half-line $(r, +\infty)$.  

We now find an elementary set $U'_0 \sbt \interior(U_0)$,
such that the (also elementary) set
$A' := \pal(U'_0)$ satisfies
$m(A \setminus A') < \del$. 
Let $f_0$ be  a continuous function
 on $\R$, supported on $U_0$, and satisfying
$f_0(x) = \varphi(\pal(x))$ for $x \in U'_0$,
and $|f_0(x)| \leq M$ for every $ x \in \R$.
Since
$\pal$ is one-to-one on $\interior (U_0)$,
we have 
$\pal(f_0) = \varphi$ on the set $A'$.

Finally we define the continuous function
 \begin{equation}
\label{eqFL2.55}
f(x) := \frac1{N} \sum_{j=1}^{N} f_j(x), \quad
f_j(x) :=  f_0(x - m_j \al).
\end{equation}
Then $f_j$ 
is supported on $U_j$, $1 \leq j \leq N$, and
hence $f$ is supported on the union
 \begin{equation}
\label{eqFL2.7}
U := U_1 \cup U_2 \cup \dots \cup U_N.
\end{equation}
Recall that $U_1, \dots, U_N$ are disjoint sets, and that 
$|f_j| \leq M$ for each $j$.  
It thus follows from \eqref{eqFL2.55} that
 $|f(x)| \leq MN^{-1} \leq \del$ for every
$x \in \R$. So property 
\ref{lemFC1.1.3} is satisfied.

Notice that 
$\pal(U_j) = \pal(U_0) = A$
for every $j$. In particular, this implies
\ref{lemFC1.1.1}.

Since
$f_j$ is a translate of $f_0$ by an integer
multiple of $\al$, we have
$\pal(f_j) = \pal(f_0)$ for each $1 \leq
j \leq N$. It follows that
$\pal(f) = \pal(f_0) = \varphi$ on $A'$.
So \ref{lemFC1.1.4} is established.

Lastly, 
$\pbe(U_j) \sbt J$ for every $j$, hence by
\eqref{eqFL2.7}
we have
$\pbe(U) \sbt J$ as well. We conclude that also
the  condition
 \ref{lemFC1.1.2} holds 
and the lemma is proved.
\end{proof}

\begin{proof}[Proof of \thmref{thmR25.1}]
The  approach is similar to the proof of
 \lemref{lemFL1.2}, so we shall be brief.
We construct by induction 
a sequence $A_1, A_2, \dots$ of  pairwise disjoint elementary sets in $\tal$, 
a sequence $B_1, B_2, \dots$ of  pairwise disjoint  elementary sets in $\tbe$, 
a sequence $U_1, V_1, U_2, V_2, \dots$  of pairwise disjoint elementary sets 
in $\R$ accumulating at infinity, 
and a sequence $g_1,h_1,g_2,h_2, \dots$ 
of continuous functions on $\R$,
in the following way.

Suppose that we have already constructed the 
sets $A_k, B_k, U_k, V_k$ and the functions $g_k, h_k$ for
 $1 \leq k \leq n-1$. We use \lemref{lemFC1.1}
to find an elementary 
set $U_n \sbt \R$, and a continuous function $g_n$ on $\R$, such that
$\pal(U_n)$ is disjoint from the sets $A_1, \dots, A_{n-1}$,
$\pbe(U_n)$ is disjoint from the sets $B_1, \dots, B_{n-1}$,
$g_n$ is supported on $U_n$, $|g_n(x)| \leq 2^{-n}$ for all $x \in U_n$,
and $\pal(g_n) = p - \sum_{k=1}^{n-1} \pal(h_k)$ on some
elementary set $A_n \sbt \tal$, which is 
disjoint from $A_1,  \dots, A_{n-1}$, and such that
 \begin{equation}
\label{eqFL3.22}
(1 - 2^{-n}) \al < m(A_1 \cup \dots \cup A_n) < \al.
\end{equation}
Then, we use again \lemref{lemFC1.1} but with the
roles of $\al,\be$ interchanged, to find an elementary 
set $V_n \sbt \R$, and a continuous function $h_n$ on $\R$, such that
$\pbe(V_n)$ is disjoint from the sets $B_1, \dots, B_{n-1}$,
$\pal(V_n)$ is disjoint from $A_1, \dots, A_n$,
$h_n$ is supported on $V_n$, $|h_n(x)| \leq  2^{-n}$ for all $x \in V_n$,
and $\pbe(h_n) = q - \sum_{k=1}^{n} \pbe(g_k)$ on some
elementary set $B_n \sbt \tbe$, which is 
disjoint from $B_1,  \dots, B_{n-1}$, and such that
 \begin{equation}
\label{eqFL3.25}
(1 - 2^{-n}) \be  < m(B_1 \cup \dots \cup B_n) < \be.
\end{equation}

We observe that  \lemref{lemFC1.1} allows us to choose the 
sets $U_1, V_1, U_2, V_2, \dots$  
to be  pairwise disjoint and accumulating at $+ \infty$.  
So we may assume this to be case.

Finally, we define 
$f  := \sum_{n=1}^{\infty} (g_n + h_n)$,
which is a continuous function on $\R$  vanishing at infinity.
Similarly to  the proof of
 \lemref{lemFL1.2}, one can check that
$\pal(f) = p$  on the union $\cup_{n=1}^{\infty} A_n$, 
a set of full measure in $\tal$, 
while
$\pbe(f) = q$
 on $\cup_{n=1}^{\infty} B_n$, 
a set of full measure in $\tbe$.
Thus $f$ satisfies the simultaneous tiling condition
\eqref{eqA2.1.1}.
(We note that both sums in \eqref{eqA2.1.1}
have only finitely many nonzero terms
for almost every $x \in \R$.)
\end{proof}

\subsection{Remarks}
1. One can choose the
function $f$ in  \thmref{thmR25.1}
to be not only
continuous  but in fact \emph{smooth}. 
To this end it suffices to replace
\lemref{lemFC1.1} with a similar version,
where 
$\varphi$ and $f$ are smooth functions.

2.  If the tiling level vector $(p,q)$ is not proportional to $(\be, \al)$,
then the function $f$ in  \thmref{thmR25.1} can only have
slow decay at infinity. In fact,   
$f$ cannot be in $L^1(\R)$ due to  \propref{propA2.1}.


\section{Incommensurable 
arithmetic progressions: Lower bounds for the support 
size of a simultaneously tiling function}
\label{secE6}

In this section we prove Theorems \ref{thmA2.3} and \ref{thmA2.4}.
These theorems give a sharp lower bound for the measure of
$\supp f$, where $f$ is an arbitrary
measurable function satisfying the simultaneous tiling condition \eqref{eqA2.1.1}.

Our proof is based on a graph-theoretic approach.
We will show that any simultaneously tiling function $f$ induces a
weighted bipartite graph, whose vertices and edges are also endowed with a 
measure space structure. The main method of the proof
is an \emph{iterative leaves removal process}
that we apply to  this graph.

Throughout this section we again
suppose that $\al, \be > 0$ are two fixed,
rationally  independent real numbers.

\subsection{Bipartite graphs and iterative leaves removal}

We start by introducing some purely graph-theoretic concepts and notation.

A \emph{bipartite graph} is a triple $G=(A,B,E)$, consisting of two disjoint sets 
$A, B$ of vertices, and a set $E\subset A \times B$ of edges.
The sets $A, B$ and $E$ may be infinite, and may even be uncountable.
However, we will assume that \emph{each vertex in the graph $G$ has finite degree}.

For any set of vertices $S\subset A\cup B$,
 we denote by $E(S)$  the set of all edges which are incident to a vertex from $S$. 

For each $k \geq 0$ we let $A_k$ be the set of vertices of degree $k$ in $A$, 
and $B_k$ be the set of vertices of degree $k$ in $B$.
In particular,  $A_1$ and $B_1$ are the sets of leaves in $A$ and $B$, respectively. 
 Note that the sets $A_k$, $B_k$ form a partition of $A \cup B$. 
 
A vertex $v \in A \cup B$ will be called a \emph{star vertex} 
if all the neighbors of $v$   in the graph are leaves. 
We denote by $A_*$ the set of star vertices which belong to $A$,
and by $B_*$ the set of star vertices that belong to $B$.

\begin{definition}[leaves removal]
\label{leavRem}
Given a bipartite graph $G=(A,B,E)$ \emph{with no isolated vertices}, we define its \emph{$A$-leaves-removed-graph} to be the graph 
 \begin{equation}
\label{eqFX1.1}
G_A = (A\setminus A_1, B \setminus B_*, E\setminus E(A_1)),
\end{equation}
that is, $G_A$ is the graph obtained from $G$ by removing all the leaves on the $A$-side (including the edges incident to those leaves) and then dropping the star vertices in $B$, which are the vertices on the $B$-side that became isolated due to the removal of all their neighbors. Similarly we define the \emph{$B$-leaves-removed graph} to be
 \begin{equation}
\label{eqFX1.2}
G_B=(A\setminus A_*, B\setminus B_1, E\setminus E(B_1)).
\end{equation}
\end{definition}

\begin{remark}
Notice that assuming $G$ to have no isolated vertices implies that the new graph $G_A$ must have no isolated vertices either.  Indeed, when we remove the leaves from $A$, the only vertices which become isolated are those in $B_*$, and we make sure to remove these vertices from $B$. Similarly, the graph $G_B$ has no isolated vertices.
\end{remark}

\begin{definition}[iterative leaves removal]
\label{iterLeavRem}
Given a bipartite graph $G=(A,B,E)$
 \emph{with no isolated vertices},
 we define its \emph{leaves-removal-graph-sequence} $G^{(n)}=(A^{(n)}, B^{(n)}, E^{(n)})$
as follows. We let $G^{(0)} = G$, and for each $n$, if $n$ is even we let
$G^{(n+1)} = (G^{(n)})_A$,
while if $n$ is odd then
$G^{(n+1)} = (G^{(n)})_B$.

In other words, the sequence is obtained by consecutive removal of leaves alternately from each side of the graph.
First we remove all the leaves from the $A$-side, as well as  the
 star vertices on the $B$-side.
 By doing so we may have created some new leaves on the $B$-side, as some vertices  in $B$ may have lost all their neighbors in $A$ but one. In the second step we remove all the leaves from the $B$-side and the  star vertices on the $A$-side.
 Then again some vertices on the $A$-side may become leaves. The process continues in a similar fashion.

Notice that if at the $n$'th step of the iterative process there are no leaves to be removed on the relevant side of the graph, then we simply obtain $G^{(n+1)}=G^{(n)}$. 
\end{definition}

\begin{definition}[weighted bipartite graph]
\label{wgtGrphDef}
We say that a bipartite graph $G=(A,B,E)$ is \emph{weighted} if it is endowed with an edge-weight function $w:E\rightarrow \C$ which assigns a complex-valued weight to each edge of the graph. 
\end{definition}

For each vertex $v\in A\cup B$, the  sum of the weights of all the 
(finitely many) edges incident to $v$
will be called the \emph{weight of the vertex $v$}.

\subsection{The graph induced by a subset of the real line}

We now  turn our attention to a  specific construction of a bipartite graph.

\begin{definition}[{the induced graph $G(\Omega)$}]
\label{gOm}
Let $\Omega$ be  an arbitrary   subset of $\R$.
We associate to $\Omega$ a 
bipartite graph $G(\Omega)$  defined as follows.
The set of vertices of the graph is the union of the
two disjoint sets 
$A=\pal(\Om)$ and $B=\pbe(\Om)$, 
which form the bipartition of the graph. The set of edges $E$
of the graph consists of all edges of the form
$e(x):=(\pal(x),\pbe(x))$ where  $x$
goes through the elements of   $\Om$.
\end{definition}

\begin{remark}
\label{omegaEdgesId}
We note that distinct points  $x,y \in \Omega$
correspond to distinct edges
$e(x)$, $e(y)$ in $E$. Indeed,
  if $e(x) = e(y)$ then we must have
  $x-y \in \al \Z \cap \be \Z$,
  which in turn implies that  $x=y$ since $\al,\be$ are rationally independent.
  Thus, the 
representation of the elements of $\Om$ as edges in the graph
 is one-to-one. In the sequel, we will often
identify edges of the graph  with elements of the set $\Omega$.
\end{remark}

\begin{definition}[{finite degrees assumption}]
We  say that a set $\Om \subset \R$ satisfies the \emph{finite degrees assumption} if each vertex in the graph   $G(\Om)$ has finite degree. This is the case if and only if  for every $x \in \R$, the  sets $\Om \cap (x + \al \Z)$ and $\Om \cap (x+ \be \Z)$ have both finitely many elements.
\end{definition}

In what follows, we shall assume that the given set $\Om \subset \R$ satisfies the finite degrees assumption.

\begin{remark}
Notice that the graph $G(\Om) = (A,B,E)$ has no isolated vertices. 
Indeed, if $a$ is a vertex in $A$ then $a = \pal(x)$ 
for some $x \in \Om$, so $a$ is incident to the edge 
$e(x)=(\pal(x),\pbe(x))$. Similarly, any vertex $b \in B$
 is incident to at least one edge in $E$.
\end{remark}

\begin{remark}
\label{remE11.2}
Let $G_A(\Om)$ be the $A$-leaves-removed-graph of $G(\Omega)$.
Notice that $G_A(\Om)$ is the graph induced by 
the set $\Omega_A = \Omega \setminus E(A_1)$, where here
 we identify edges  of the graph with elements of the set $\Omega$ (see
 \remref{omegaEdgesId}). Thus we have $G_A(\Omega) = G(\Omega_A)$.
Similarly,  the $B$-leaves-removed-graph $G_B(\Om)$  of $G(\Omega)$ is the 
 graph induced by the set $\Omega_B = \Omega \setminus E(B_1)$
 (where again edges of the graph are identified with elements of $\Om$).
Hence the iterative leaves removal process applied to the graph $G(\Omega)$
induces a sequence of sets
 $\Omega^{(n)} \sbt \R$,  satisfying
 $\Omega^{(n+1)} \subset \Omega^{(n)} \subset \Om$ for all $n$, and
 such that   the leaves-removal-graph-sequence
 $ G^{(n)}(\Omega)$ is given by
$G^{(n)}(\Omega) = G(\Omega^{(n)})$.
\end{remark}

\subsection{Vertices and edges as measure spaces}
Assume now that  $\Om$ is a \emph{measurable} subset of $\R$,
satisfying  the finite degrees assumption. 
In this case the induced graph $G(\Om)$
can be endowed with an additional measure space structure, as follows.

Recall that we have endowed  $\T_\al$ and $\T_\be$ with the Lebesgue measure, normalized such that $m(\T_\al) = \al$ and $m(\T_\be) = \beta$. We notice that the two vertex sets $A = \pal(\Om)$ and $B = \pbe(\Om)$ of the graph $G(\Om) = (A,B,E)$ are measurable subsets of $\tal$ and $\tbe$ respectively. We may therefore consider $A$ and $B$ as measure spaces, with the measure space structure induced from the embedding  of $A$ and $B$ into $\tal$ and $\tbe$ respectively.

We also endow the edge set $E$ with a measure space  structure, induced from the identification of $E$ with $\Om$ as a (measurable) subset of $\R$ as in  \remref{omegaEdgesId}. (Notice that we \emph{do not}  endow 
$E$ with the measure  space structure induced from the embedding of $E$ into the product space $A \times B$.) 

 In the sequel, we will also consider the entire vertex set  $V := A \cup B$ as a single measure space, formed by the disjoint union of the two measure spaces $A$ and $B$.

\begin{lem}[measurability lemma]
\label{mesLem}
\quad
\begin{enumerate-roman}
\item
   \label{mesLem.1}
For each $k$ the set $A_k$ of vertices of degree $k$ in $A$ is a measurable subset of $A$;
\item
   \label{mesLem.2}
The set $A_*$ of star vertices in $A$ (that is, the
 vertices in $A$ all of whose neighbors are leaves) is a measurable subset of $A$;
\item
   \label{mesLem.3}
If $S \subset A$ is a measurable set of vertices, then $E(S)$ (the set of edges incident to a vertex in $S$) is a measurable subset of $E$.
\end{enumerate-roman}
Similar assertions hold for the sets $B_k$, $B_*$ and $S \sbt B$.
\end{lem}

\begin{proof}
If $a$ is a vertex in $A$, then 
the degree of $a$ in the graph $G(\Om)$ is equal to $\pal(\1_\Om)(a)$.
Hence $\pal(\1_\Om)$ is an everywhere finite, measurable function on $\tal$.
Since for each $k$ the set $A_k$ 
 is the preimage of $\{k\}$ under this function,  it follows that $A_k$ is measurable.
 
By a similar argument, also the set $B_k$ is measurable  for each $k$.

 Next we observe that
  \begin{equation}
\label{AsBsMes}
A_* = A\setminus \pal(\Om \cap \pbe^{-1}(B \setminus B_1 )),
\quad 
B_* = B\setminus \pbe(\Om \cap \pal^{-1}(A \setminus A_1 )),
\end{equation}
hence both sets $A_*$ and $B_*$ are measurable.

Finally, let $S$ be  a measurable subset of $A$. 
Identifying the edges  of the graph 
with elements of $\Omega$, we have
$E(S) = \pal^{-1}(S) \cap \Om$, hence $E(S)$ is measurable.
Similarly, for any measurable   set
$S \sbt B$, the set $E(S) = \pbe^{-1}(S) \cap \Om$
 is measurable.
\end{proof}

\begin{remark}
Recall from \remref{remE11.2} that
 the iterative leaves removal process 
 induces a sequence of sets
 $\Omega^{(n)} \sbt \R$,  satisfying
 $\Omega^{(n+1)} \subset \Omega^{(n)} \subset \Om$ for all $n$, and
 such that   the leaves-removal-graph-sequence
 $ G^{(n)}(\Omega)$ is given by
$G^{(n)}(\Omega) = G(\Omega^{(n)})$.
It follows from \lemref{mesLem}
that if $\Om$ is a measurable subset of $\R$, then all the sets
 $\Omega^{(n)}$ are measurable too,
 since the set of edges removed at each step of the
 iterative process   is measurable.
\end{remark}

For a vertex $a \in A$ we denote by
 $\deg_A(a)$ the degree of $a$
in the graph $G(\Om)$. Similarly, we 
denote by $\deg_B(b)$  the degree of a vertex $b \in B$.
Then 
 $\deg_A$ and $\deg_B$ are nonnegative,
 integer-valued functions on $A$ and $B$ respectively.

\begin{lem}[edge counting lemma]
\label{mesLemRN}
$\deg_A$ is a measurable function on $A$.
Moreover, 
 for any measurable set of vertices $S \subset A$ we have
 \begin{equation}
\label{eqFX2.1}
m(E(S)) = \int_S \deg_A.
\end{equation}
In particular,
 \begin{equation}
\label{eqFX2.2}
m(E(A_k)) = k \cdot m(A_k), \quad  k=1,2,3, \dots
\end{equation}
Similar assertions hold for $\deg_B$ and $B_k$.
\end{lem}

Notice that the integral in \eqref{eqFX2.1} may be finite or infinite,
but in any case it has a well-defined value,
since $\deg_A$ is a nonnegative function.

\begin{proof}[Proof of \lemref{mesLemRN}]
Let $S \subset A$ be a measurable set.
By identifying the edges  of the graph $G(\Om)$
with elements of $\Omega$, we have
$E(S) = \pal^{-1}(S) \cap \Om$. Then
 \begin{equation}
\label{eqFX15.1}
m(E(S)) = \int_{\R} \1_{E(S)} = \int_{\tal} \pal(\1_{E(S)})
\end{equation}
(these equalities hold both if $E(S)$ has finite or infinite measure).
But notice that for a vertex $a \in A$ we have
 \begin{equation}
\label{eqFX15.2}
\pal(\1_{E(S)})(a) = 
\begin{cases}
\deg_A(a), & a \in S, \\[4pt]
0, & a \notin S.
\end{cases}
\end{equation}
Thus \eqref{eqFX15.1} and
\eqref{eqFX15.2}  imply \eqref{eqFX2.1}.
Finally \eqref{eqFX2.2} is a consequence
of \eqref{eqFX2.1}, since the function
$\deg_A$ attains the constant value $k$ on the set $A_k$.
\end{proof}

\begin{remark}
Let $\mu_A$ be the measure on $A$ obtained as the image under the map $\pal$
of the Lebesgue measure restricted to $\Om$. The
assertion of \lemref{mesLemRN} may be equivalently 
stated by saying that
$\deg_A$ is the Radon-Nikodym derivative of $\mu_A$
with respect to the Lebesgue measure on $A$.
\end{remark}

\subsection{A brief digression: Measure preserving graphs (graphings)}

The graph $G(\Om)$ endowed with its measure space structure
is closely related to the notion of a \emph{measure preserving graph}, or a \emph{graphing},
so we will discuss this relation briefly here.
For  a detailed exposition we  refer to the book 
by Lov\'{a}sz \cite[Chapter 18] {Lov12}.

A \emph{Borel graph} is a graph $(V,E)$ where the vertex set $V$ is a standard Borel space (i.e.\ the measure space associated to a separable, complete metric space), and the edge set $E$ is a Borel set in 
$V \times V$. One can show that if
$\Om \sbt \R$ is a Borel set, then the induced graph $G(\Om)$ is a Borel graph.

A \emph{measure preserving graph}, or a  \emph{graphing}, is a Borel graph $(V,E)$  whose vertex set $V$ is endowed with a probability measure $\lambda$, such that for any two measurable sets $U,W \sbt V$ we have
 \begin{equation}
\label{eqFX3.1}
\int_U n_W(x)d\lambda(x) = \int_W n_U(x)d\lambda(x),
\end{equation}
where $n_U(x)$ and $n_W(x)$ denote the number of neighbors of $x$ within the sets $U$  and $W$ respectively.
The last condition relates the graph structure to the measure space structure by requiring that ``counting'' the edges between $U$ and $W$ from $U$, yields the same result as counting them from  $W$.
One can show based on \lemref{mesLemRN} that the graph $G(\Om)$
satisfies the condition \eqref{eqFX3.1}.

We point out however that in \cite{Lov12}  the notion of a graphing includes the additional assumption that the degrees of the vertices in the graph are bounded by a certain constant. To the contrary, for the graph $G(\Om)$ we only assume that each vertex has finite degree, allowing the existence of vertices with arbitrarily large degrees.

\subsection{The graph induced by a simultaneously tiling function}

Let now $f$ be a measurable function on $\R$, and
consider the graph $G(\Om)=(A,B,E)$ induced
by the set $\Om := \supp(f)$.
By identifying the edges  of the graph
with elements of the set $\Omega$ (as in
 \remref{omegaEdgesId}) we may view 
 $f$ as a  function on the set of edges of the graph.
Thus  $G(\Omega)$  becomes a weighted graph,
 with the weight function $f$.

\begin{lem}
  \label{lemFinDegTile}
Let $f$ be a measurable function on $\R$, $\mes(\supp f) < + \infty$, and assume that
 $f$ satisfies the simultaneous tiling condition \eqref{eqA2.1.1}.
Then $f$ can be redefined on a set of measure zero
so as to satisfy also the following two additional properties:

\begin{enumerate-roman}
\item
  \label{lemK1.1.1}
The set  $\Om := \supp f$
satisfies the finite degrees assumption;

\item
  \label{lemK1.1.2}
  If the induced   graph $G(\Om)=(A,B,E)$ is weighted by the function $f$, then
each vertex from $A$   has weight $p$, while each vertex from $B$  has weight $q$.
\end{enumerate-roman}
\end{lem}

\begin{proof}
Denote the given function by $f_0$, and let
$\Om_0 := \supp f_0$. 
Let $X_0$ be the set of all points $x \in \R$ satisfying  
the   conditions
\begin{equation}
\label{eqK1.1.3}
\sum_{k \in \Z} \1_{\Om_0}(x-k\al) < +\infty,
\quad
\sum_{k \in \Z}f_0(x-k\al)=p,
\end{equation}
as well as the conditions 
\begin{equation}
\label{eqK1.1.3.5}
\sum_{k \in \Z}\1_{\Om_0}(x-k\be) < +\infty,
\quad
\sum_{k \in \Z}f_0(x-k\be)=q.
\end{equation}
The assumptions  imply that $X_0$ is a set of full measure in $\R$. 
Then also the set 
\begin{equation}
\label{eqK1.1.4}
X := \bigcap_{(n,m) \in \Z^2} (X_0 + n\al + m\be)
\end{equation}
 has full measure in $\R$. We define
 $f := f_0 \cdot \1_X$, then $f$ coincides with $f_0$ a.e.
We will show that the new function $f$ satisfies  the two additional conditions 
\ref{lemK1.1.1} and \ref{lemK1.1.2}.

Let $G(\Om)=(A,B,E)$  be the graph induced by the set 
 $\Om := \supp f = \Om_0 \cap X$.
 We first verify the condition 
\ref{lemK1.1.1}, namely, we show that  each vertex of $G(\Om)$
has finite degree. Indeed, let $a \in A$, then $a=\pal(x)$
 for some $x \in \Om$, and the degree of $a$ is the 
 number of elements in the set $\Om \cap (x + \al \Z)$.
 But this set has finitely many elements, which follows
 from the first condition in \eqref{eqK1.1.3} using the fact that
 $\Om \sbt \Om_0$ and $x \in \Om \sbt X_0$.
 Hence each vertex $a \in A$ has finite degree in the graph $G(\Om)$. 
 Similarly, each vertex $b \in B$  also has finite degree.

Now let the graph $G(\Om)$ be weighted by the function $f$.
We show that condition \ref{lemK1.1.2} holds. Indeed,
let $a \in A$, then again $a=\pal(x)$
 for some $x \in \Om$. Since $\Om \sbt X$ and 
the set $X$ is  invariant under translations by elements
from $\al \Z$, we have
 $x + \al \Z \sbt X$ and  thus $f$ coincides
 with $f_0$ on the set 
  $x + \al \Z$. This implies that
 $\pal(f)(a) = \pal(f_0)(a) = p$,
where the last equality  follows from the second
condition in  \eqref{eqK1.1.3} using the fact that
$x \in \Om \sbt X_0$.
But  $\pal(f)(a)$ is exactly the weight of the vertex $a$ in the graph $G(\Om)$,
 hence the vertex $a$ has weight $p$.
 The proof that each 
vertex $b \in B$ has weight $q$ is similar.
\end{proof}

In what follows, we  assume that 
$f$ is a measurable function on $\R$
satisfying the simultaneous tiling condition \eqref{eqA2.1.1}.
Since our goal is to obtain a lower bound for the 
measure of the support of $f$, we assume that 
$\Om := \supp f$ is a set of finite measure.

We endow the graph $G(\Om)=(A,B,E)$ 
with the weight function $f$. By redefining the
 values of  $f$ on a set of measure zero
 (using  \lemref{lemFinDegTile})
we can assume with no loss of generality  that 
every vertex in the graph has finite degree,
and that the vertices from $A$   have weight $p$,
 while the vertices from $B$  have weight $q$.

We will also assume that the tiling levels $p$ and $q$ 
in  \eqref{eqA2.1.1} are both nonzero 
(the case where one of $p,q$ is zero is covered by
\thmref{thmA2.6}). This implies that the supports of 
the functions
$\pal(f)$ and $\pbe(f)$  coincide with $\tal$ 
and $\tbe$ respectively 
up to a set of measure zero. Hence  
\begin{equation}
 \label{eqFX4.1}
m(A) = \al, \quad m(B) = \be.
\end{equation}

\subsection{The Euler characteristic}

Recall that the set $E$ of edges of the graph $G(\Om)$ is endowed with a
measure space structure, induced from 
the identification of $E$ with $\Om$ as a measurable subset of $\R$
(\remref{omegaEdgesId}).  In particular, $m(E) = m(\Om) < + \infty$.

\begin{definition}[Euler characteristic]
The quantity
\begin{equation} \label{eqFX5.1}
\euler    =  m(A)+m(B)-m(E)
\end{equation}
will be called the \emph{Euler characteristic} of the graph  $G(\Om)=(A,B,E)$.
\end{definition}

We call this quantity the ``Euler characteristic'' since
it is the difference between the total
measure of the vertices in the graph and the total measure of the edges.

Similarly, we let
\begin{equation} \label{eqFX5.2}
\euler^{(n)}  = m(A^{(n)})+m(B^{(n)})-m(E^{(n)})
\end{equation}
 denote the Euler characteristics of the
leaves-removal-graph-sequence $ G^{(n)}(\Omega)$.

Let ${L}^{(n)}$ be the set of leaves removed at the $n$'th step
of the iterative leaves removal process, that is, if $n$ is even then ${L}^{(n)} = A^{(n)}_1$ (the set of leaves in $A^{(n)}$), and if $n$ is odd then ${L}^{(n)} = B^{(n)}_1$ (the set of leaves in $B^{(n)}$). 
The next lemma gives a lower bound for 
the measure of the set  ${L}^{(n)}$ in terms of the Euler characteristic $\euler^{(n)}$.

\begin{lem}[removed leaves estimate]
\label{leavesBound} 
Assume that $\al > \be$. Then
\begin{equation} \label{eqFX5.5}
m({L}^{(0)}) > \euler^{(0)},
\end{equation}
and for all $n \geq 1$ we have
\begin{equation} \label{eqFX5.6}
m({L}^{(n)}) \geq 2\euler^{(n)}.
\end{equation}
\end{lem}

The assumption   that $\al > \be$ 
can be made with no loss of generality,
for otherwise we may simply interchange the roles of  $\al$ and $\be$.
The reason we   need to make this assumption  is that we 
have chosen to  begin  the iterative leaves removal process by
removing leaves from the $A$-side.
(If we had $\be > \al$ then the process 
would have to begin  by
removing leaves from the $B$-side.)

To prove \lemref{leavesBound}  we will first establish two additional lemmas. 
The first one gives a lower bound for the measures
 of the sets of leaves $A_1$ and $B_1$.

\begin{lem}
  \label{lemJ1.1}
  We have
  \begin{equation} \label{eqFX5.8}
  m(A_1) \geq 2 m(A)-m(\Omega),
\end{equation}
  and similarly,
  \begin{equation} \label{eqFX5.9}
  m(B_1) \geq 2 m(B)-m(\Omega).
\end{equation}
\end{lem}

\begin{proof}
Recall that we denote by $A_k$  the set of vertices in $A$ of degree $k$. 
Since the sets $A_k$   form a partition of $A$,  we have
\begin{equation}
\label{eqJ1.1.11}
m(A) = \sum_{k=1}^{\infty} m(A_k).
\end{equation}
In turn, the
 sets $E(A_k)=\pal^{-1}(A_k) \cap \Omega$ 
 form a partition of $\Omega$, and by  \lemref{mesLemRN}
we have $m(E(A_k))=  k   m(A_k)$. Hence
\begin{equation}
\label{eqJ1.1.12}
m(\Omega) = \sum_{k=1}^{\infty} m(E(A_k)) = \sum_{k=1}^{\infty}k m(A_k).
\end{equation}
Using \eqref{eqJ1.1.11} and \eqref{eqJ1.1.12} we conclude that
\begin{equation}
\label{eqJ1.1.3}
2 m(A) - m(A_1) 
= m(A_1) + 2 \sum_{k=2}^{\infty} m(A_k)
\leq \sum_{k=1}^{\infty} k m(A_k) = m(\Om),
\end{equation}
which proves \eqref{eqFX5.8}. The inequality \eqref{eqFX5.9} can 
be proved  in a similar way.
\end{proof}

The next lemma is a more symmetric version of the previous one.

\begin{lem}
\label{lemLeavsExist}
We have
\begin{equation}
\label{eqJ1.1.22}
m(A_1) + m(B_1) \geq 2 \euler.
\end{equation}
\end{lem}

In other words, the measure of the set of leaves
 in the graph $G(\Om)$,  both from $A$ and from $B$, 
 is at least $2\euler$. 
This is an immediate consequence of \lemref{lemJ1.1}. Indeed, 
taking the sum of \eqref{eqFX5.8} and \eqref{eqFX5.9} yields
\begin{equation}
\label{eqJ5.1}
m(A_1) + m(B_1) \geq 2m(A) - m(\Om) + 2m(B) - m(\Om) =2\euler.
\end{equation}

Now we can prove \lemref{leavesBound} based on the previous two lemmas.

\begin{proof}[Proof of \lemref{leavesBound}]
Recall from \eqref{eqFX4.1} that we have
$m(A) = \al$, $m(B) = \be$, and  that we
have assumed $\al > \be$. Hence using \lemref{lemJ1.1}
 we obtain
 \begin{equation}
m({L}^{(0)}) \geq 2\al - m(\Om)> \al +\be -m(\Om) =\euler^{(0)},
 \end{equation}
and so \eqref{eqFX5.5} is proved. Next,
for  $n\geq 1$ we  apply \lemref{lemLeavsExist} to 
the graph $G^{(n)}(\Om)$. The lemma gives
 \begin{equation}
\label{eqFX5.25}
m(A^{(n)}_1) + m(B^{(n)}_1) \geq 2 \euler^{(n)}.
 \end{equation}
However we observe that for $n \geq 1$, the set of leaves
$B^{(n)}_1$ is empty if $n$ is even, and
$A^{(n)}_1$ is empty if $n$ is odd, 
due to the removal of the leaves  in 
the previous step of the iterative leaves removal process. 
Hence \eqref{eqFX5.6} follows from
\eqref{eqFX5.25}.
\end{proof}

\begin{lem}[monotonicity]
\label{lemMonoton}
For every $n$ we have
 \begin{equation}
\label{eqFX5.17}
\euler^{(n+1)}\leq \euler^{(n)}.
 \end{equation}
\end{lem}

\begin{proof}
Suppose first that $n$ is even. By
the definitions of $\euler^{(n+1)}$ and  $G^{(n+1)}(\Om)$ we have
\begin{align}
 \euler^{(n+1)} &= m(A^{(n+1)}) + m(B^{(n+1)}) - m(E^{(n+1)}) \nonumber  \\
& =(m(A^{(n)}) - m(A^{(n)}_1)) + (m(B^{(n)}) - m(B^{(n)}_*)) - (m(E^{(n)}) - m(E^{(n)}(A^{(n)}_1)))  \nonumber  \\
& = \euler^{(n)} - m(A^{(n)}_1) - m(B^{(n)}_*) + m(E^{(n)}(A^{(n)}_1) ) =
 \euler^{(n)} - m(B^{(n)}_*),
\end{align}
where in the last equality we used 
$m(E^{(n)}(A^{(n)}_1)) = m(A^{(n)}_1)$ 
(\lemref{mesLemRN}).
Hence for even $n$ we have
\begin{equation}
\label{etanEven}
\euler^{(n+1)} =\euler^{(n)} - m(B^{(n)}_*).
\end{equation}
Similarly, for odd $n$ we have
\begin{equation}
\label{etanOdd}
\euler^{(n+1)} =\euler^{(n)} - m(A^{(n)}_*),
\end{equation}
and the inequality \eqref{eqFX5.17} follows.
\end{proof}

\subsection{Jump sets and measure jumps}

Let us denote by $J^{(n)}$ the set $B^{(n)}_*$ if $n$ is even, or the set $A^{(n)}_*$ if $n$ is odd.
The set $J^{(n)}$ will be called a \emph{jump set}.
The equalities \eqref{etanEven} and \eqref{etanOdd}, established in 
the proof of \lemref{lemMonoton}, say that for every $n$ we have
\begin{equation}
\label{eqgapmJn}
m(J^{(n)}) = \euler^{(n)} - \euler^{(n+1)}.
\end{equation}

\begin{definition}[measure jump]
Whenever it happens that $\euler^{(n)} > \euler^{(n+1)}$,
or equi\-valently, whenever  we have 
$m(J^{(n)}) > 0$, 
 we will say that a \emph{measure jump} has occurred.  
\end{definition}

\begin{lem}[finite subtree lemma]
\label{lemJump} Assume that for some $n$ the set $J^{(n)}$ is nonempty (in particular, this is the case if $\euler^{(n)} > \euler^{(n+1)}$). Then for each vertex $v\in J^{(n)}$, the connected component of $v$ in the  graph $G(\Om)$ is a finite tree. Moreover, if $v,w$ are two
distinct vertices in $J^{(n)}$ then their respective connected components in $G(\Om)$ are disjoint.
\end{lem}

\begin{proof}
By definition, $J^{(n)}$ is either $B^{(n)}_*$ or $A^{(n)}_*$ depending on whether $n$ is even or odd.
We consider the case where $n$ is even (the case where $n$ is odd is similar). Then $J^{(n)} = B^{(n)}_*$ is the set of star vertices in $B^{(n)}$, that is, the
vertices in $B^{(n)}$ all of whose neighbors in the graph $G^{(n)}$ are leaves. 
Recalling that all the degrees of vertices in $G(\Om)$ are finite, it follows that
the connected component of a vertex $v \in J^{(n)}$ in the graph $G^{(n)}$ is a finite tree
(of height one, if we view $v$ as the root of the tree). Since
$G^{(n)}$ was obtained from $G^{(n-1)}$ by leaves removal,
the connected component of $v$ in the graph $G^{(n-1)}$ is again a finite tree
(of height at most two, when counted from the root $v$). 
Continuing in the same way, we conclude that
the connected component of 
$v$ in the graph $G^{(n-j)}$ is a finite tree (of height at most $j+1$
from the root $v$),
for each $j=1,2,\dots,n$. In particular, for $j=n$ we obtain 
the first assertion of the lemma.

Next we turn to prove 
the second assertion of the lemma. Consider two distinct vertices $v,w \in J^{(n)}$. 
Since all the neighbors of both $v$ and $w$ within $G^{(n)}$ are leaves,
$v$ and $w$ cannot have any common neighbor in $G^{(n)}$.
Hence  the connected components of $v$ and $w$ in the graph 
$G^{(n)}$ are disjoint. Since
$G^{(n)}$ was obtained from $G^{(n-1)}$ by leaves removal,
the connected components of $v$ and $w$ in the graph $G^{(n-1)}$ 
are also disjoint.
Continuing in the same way, we conclude that
the connected components of $v$ and $w$ 
 in the graph $G^{(n-j)}$ are disjoint
for each $j=1,2,\dots,n$. In particular this is the case for
$j=n$ and so the second assertion of the lemma follows.
\end{proof}

 \subsection{Proof of \thmref{thmA2.3}}

We now turn to prove \thmref{thmA2.3}.
Recall that the theorem asserts that if
the tiling level vector $(p,q)$ 
 is not proportional to any vector of
the form $(n,m)$ where $n,m$ are  nonnegative integers,
then $m (\Om) \geq \al + \be$.
To prove this result,
we will assume that the tiling levels $p$ and $q$ are both nonzero  and that
\begin{equation} 
\label{eq.omAlBe}
m(\Omega) < \alpha +\beta,
\end{equation}
and we will show that this  implies that
$(p,q)$ must be 
proportional to some vector of
the form $(n,m)$ where $n,m$ are two positive integers.

Recall from \eqref{eqFX4.1} that  we have
$m(A) = \al$, $m(B) = \be$, hence
\eqref{eq.omAlBe}  is equivalent to  the assumption that
\begin{equation} 
\label{eqFX5.37}
m(E) < m(A) + m(B),
\end{equation}
that is, the total measure of the edges in the graph $G(\Om)$ is strictly smaller than 
the total measure of the vertices.

The following lemma shows that to prove \thmref{thmA2.3}
 it would be enough to establish the existence of 
a finite connected component in the graph $G(\Om)$.

\begin{lem}[total weight equality]
\label{DoubleCount}
Assume that the graph $G(\Om)$ has a finite connected component $H$,
and suppose that $H$ has
$m$ vertices in $A$ and $n$ vertices in $B$.  Then
\begin{equation} 
\label{eqFX5.64}
mp=nq.
\end{equation}
\end{lem}

\begin{proof}
Recall that we have assumed (using \lemref{lemFinDegTile}) that the  weight of each
vertex in $G(\Om)$ is either $p$ or $q$, depending on whether this vertex lies in $A$ or in $B$.
Consider the total weight of the connected component $H$, that is,
the sum of the weights of all the edges in $H$.  On   
one hand, this sum is the same as the sum of the  weights of the vertices of $H$ 
that belong to  $A$, 
 and therefore it is equal to $mp$. On the other hand, this sum is also
the same as the sum of the  weights of the vertices 
of $H$ that belong to  $B$, 
 so it must also be equal to $nq$. Hence the equality in \eqref{eqFX5.64} must hold.
\end{proof}

We can now complete the proof of \thmref{thmA2.3}.

\begin{proof}[Proof of  \thmref{thmA2.3}]
We may assume with no loss of generality that  $\al > \be$.
Consider the graph $G(\Om)$ and its leaves-removal-graph-sequence $G^{(n)}(\Om)$.
Suppose that \eqref{eq.omAlBe} holds, then equivalently we have
\eqref{eqFX5.37} and thus
$\euler^{(0)} > 0$. We claim that after at most $r = \lceil\frac{\al+\be}{\euler^{(0)}}\rceil$ 
steps of the iterative leaves removal process, a measure jump must occur.
Indeed, if not then $\euler^{(n)} = \euler^{(0)}$
for each $0 \leq n \leq r$.
 But then \lemref{leavesBound} implies that
 the measure of the set  ${L}^{(n)}$  
 of the leaves removed 
  at the $n$'th step  of the iterative process, 
  is at least $\euler^{(0)}$ 
  for each $0 \leq n \leq r$.
  Thus the total measure of the removed  leaves must be
 at least $(r+1) \euler^{(0)}$. But $(r+1) \euler^{(0)}$ is greater
 than $m(A) + m(B)$ which is the total measure of the set of 
vertices in the entire graph $G(\Om)$, so we   arrive at  a contradiction.
 Hence a measure jump must occur.

We thus conclude that there exists at least one
jump set $J^{(n)}$ of positive measure,
that is, there is $n$ such that $m(J^{(n)}) = \euler^{(n)} - \euler^{(n+1)} > 0$.
In particular the jump set $J^{(n)}$ is nonempty. Then by
 \lemref{lemJump}, any vertex $v\in J^{(n)}$ belongs to a finite
 connected component of the graph $G(\Om)$.
 Thus $G(\Om)$ has a finite connected component. 
By \lemref{DoubleCount}, there exist two positive integers $n,m$ such that
$m  p= n q$.  We conclude that
the vector $(p,q)$ is proportional to  $(n,m)$. \thmref{thmA2.3} is thus proved.
\end{proof}

\subsection{The total jump set}
We now move on towards our next  goal, which
 is to prove \thmref{thmA2.4}. This will require a more 
 detailed analysis of the jump sets which occur in our 
 iterative leaves removal process.   
 We start with following lemma.

\begin{lem}[Euler characteristics limit]
\label{lemLim}
We have
\begin{equation} 
\label{eqFX5.69}
\lim_{n\rightarrow\infty} \euler^{(n)} \leq 0.
\end{equation}
\end{lem}

Notice that the existence of the limit in \eqref{eqFX5.69}
is guaranteed due to the 
monotonicity of  the sequence
$\euler^{(n)}$  (\lemref{lemMonoton}).

\begin{proof}[{Proof of \lemref{lemLim}}]
Let
$G^{(n)}(\Om) = (A^{(n)}, B^{(n)}, E^{(n)})$  be 
the leaves-removal-graph-sequence of $G(\Om) = (A,B,E)$,
and recall that
\begin{equation} 
\label{eqFX5.71}
A^{(n+1)} \sbt A^{(n)}, \quad
B^{(n+1)} \sbt B^{(n)}, \quad
E^{(n+1)} \sbt E^{(n)},
\end{equation}
since the graph $G^{(n+1)}$ is obtained from $G^{(n)}$ by the removal of
vertices and edges. 
We define the \emph{graph limit} 
$G^{(\omega)}(\Omega) = (A^{(\omega)}, B^{(\omega)}, E^{(\omega)})$ by
\begin{equation} 
\label{eqFX5.72}
A^{(\omega)} = \bigcap_n A^{(n)}, \quad  
B^{(\omega)} = \bigcap_n B^{(n)}, \quad
 E^{(\omega)} = \bigcap_n E^{(n)}.
\end{equation}
Equivalently, 
$G^{(\omega)}(\Omega)$  is the graph induced by the 
set $\Omega^{(\omega)} = \bigcap_n  \Omega^{(n)}$ 
(the equivalence  can be verified in a straightforward way using the
finite degrees assumption).

It follows from
\eqref{eqFX5.71} and  \eqref{eqFX5.72} that
\begin{equation} 
\label{eqFX5.74}
m(A^{(n)}) \to m(A^{(\omega)}), \quad
 m(B^{(n)}) \to m(B^{(\omega)}), \quad
m(E^{(n)}) \to m(E^{(\omega)}),
\end{equation}
and consequently,
 \begin{align}
 \lim_{n\rightarrow \infty} \euler^{(n)} &= \lim_{n\rightarrow \infty} 
 (m(A^{(n)}) + m(B^{(n)}) - m(E^{(n)}) ) \nonumber \\
& = m(A^{(\omega)}) + m(B^{(\omega)}) - m(E^{(\omega)})
= \euler^{(\omega)}, \label{eqLimEuler}
\end{align}
where $\euler^{(\omega)}$ is the 
Euler characteristic of the graph 
$G^{(\omega)}(\Omega)$.

 Now, suppose to the  contrary that $\euler^{(\omega)} > 0$. Then
 we may apply \lemref{lemLeavsExist}  to  the graph limit  $G^{(\omega)}(\Om)$
  and obtain that there must be a set of leaves with positive measure in $G^{(\omega)}(\Om)$. 
  Let $v$ be any leaf of  $G^{(\omega)}(\Om)$, then 
  $v$ has exactly one neighbor $w_0$ in
$G^{(\omega)}(\Om)$. Notice that the vertex $v$ must have at least one more
neighbor in the original graph $G(\Om)$, for otherwise $v$ is a leaf
in $G(\Om)$ and should have been removed in either the first or second
step of the leaves removal process.
Let  $w_0, w_1, \dots, w_k$ be all the neighbors 
of $v$ in $G(\Om)$ (there can be only finitely many  neighbors
due to the finite degrees assumption).
  Since the vertices $w_1, \dots, w_k$ are no longer in 
  the graph limit
  $G^{(\omega)}(\Om)$,
      for each $1 \le j \le k$ 
    there is $n_j$ such that
  $w_j$ is not in $G^{(n_j)}(\Om)$.
Hence if we let $N := \max\{n_1, \dots, n_k\}$
then   $G^{(N)}(\Om)$ does not contain
any one of the vertices $w_1, \dots, w_k$.
Thus $v$ is a  leaf already in 
  the graph $G^{(N)}(\Om)$. But then
  $v$ should have been removed at the $N$'th 
step of the leaves removal process, so $v$ cannot belong to the graph limit
  $G^{(\omega)}(\Om)$. We thus arrive at a contradiction. This shows that
  $\euler^{(\omega)}$ cannot be positive and the lemma is proved.
\end{proof}

\begin{definition}[the total jump set]
The set
\begin{equation}
\label{eqtotalJS}
J = \bigcup_{n=0}^{\infty} J^{(n)}
\end{equation}
will be called the \emph{total jump set} of the graph $G(\Om)$.
\end{definition}

Recall that $J^{(n)}$ is a subset of $B$ 
if $n$ is even,  and 
$J^{(n)}$ is a subset of $A$   if $n$ is odd.
Hence $J$ is a subset of the 
entire vertex set  $V =A \cup B$ of the graph $G(\Om) = (A,B,E)$. 
Moreover, if we consider
$V$ as a measure space, formed by the disjoint union of the 
two measure spaces $A$ and $B$, then 
$J$ is a measurable subset of $V$
(\lemref{mesLem}).

We also notice that 
the sets $J^{(n)}$ form a partition of $J$
(being disjoint sets) and hence 
the measure $m(J)$ of the total jump set 
 is equal to the sum of all the measure jumps.
 By the proof of \thmref{thmA2.3}
  we know that at least one measure jump must occur, 
  which implies that the set $J$ has positive measure.
  Now we prove a stronger result:

\begin{lem}[lower bound for the total jump measure]
\label{lemJMlower}
We have
\begin{equation}
\label{eqtotalJMlower}
m(J) \geq m(A) + m(B) - m(\Om).
\end{equation}
\end{lem}

\begin{proof}
Using \eqref{eqLimEuler} we have
\begin{align}
\euler^{(0)} - \euler^{(\omega)} &= \lim_{n \to \infty}  (\euler^{(0)} - \euler^{(n)} ) 
= \lim_{n \to \infty} \sum_{k=0}^{n-1}  (\euler^{(k)} - \euler^{(k+1)} )  = \\
&= \lim_{n \to \infty} \sum_{k=0}^{n-1}  m(J^{(k)}) = m(J).
\end{align}
But due to \lemref{lemLim} we know that $\euler^{(\omega)}$ is non-positive, thus
\begin{equation}
m(J) = \euler^{(0)} - \euler^{(\omega)}  \geq 
\euler^{(0)} 
\end{equation}
which establishes \eqref{eqtotalJMlower}.
\end{proof}

\begin{lem}[total jump set  as a set of representatives]
\label{JRep}
Every connected component 
 of the graph $G(\Om)$ which is a finite tree, 
intersects the total jump set $J$ at exactly one vertex. 
Conversely, each vertex $w \in J$ lies in 
a connected component of 
 the  graph $G(\Om)$ which 
is a finite tree.
\end{lem}

Thus, we may consider the total jump set $J$ as a set of representatives,
 containing a unique representative vertex for each connected component of 
 the  graph $G(\Om)$ which 
is a finite tree.

\begin{proof}[Proof of \lemref{JRep}]
Recall that $G^{(n+1)}(\Om)$
  is obtained from $G^{(n)}(\Om)$ by (i) the removal
  of  the leaves in $A^{(n)}$ if $n$ is even, or
  the leaves in $B^{(n)}$ if $n$ is odd; (ii)
  the removal of the edges incident to the leaves removed; 
  and (iii) the removal of
   the set $J^{(n)}$ of vertices that become isolated 
   (which is the set $B^{(n)}_*$ if $n$ is even, or the set $A^{(n)}_*$ if $n$ is odd).

Now let $H$ be a connected component 
 of the graph $G(\Om)$, and assume that $H$ is a finite tree.
Then  the iterative leaves removal process necessarily exhausts the tree $H$
after a finite number of steps (this can be easily proved by 
induction on the size of the tree).  Moreover, the tree $H$ gets exhausted at the unique
step $n$ for which $J^{(n)} \cap H$ is nonempty, and
$J^{(n)} \cap H$ must then consist of exactly one vertex.

(It is worth mentioning that at the last step $n$ when 
the tree  gets exhausted, it may happen that there is only one edge
of the tree left to be removed. In this case,  one of the vertices of this edge
will be considered as a leaf, while the other vertex will be considered 
as an element of the set $J^{(n)}$.)

Thus, each 
 connected component 
 of the graph $G(\Om)$ which is a finite tree,
 contributes exactly one vertex to the total jump set $J$.

Conversely, consider a vertex $w \in J$. Then $w$
 belongs to the set $J^{(n)}$ for some $n$,
 so  by \lemref{lemJump}  the connected component of $w$
  in the graph $G(\Om)$ is a finite tree.
\end{proof}

\begin{lem}[upper bound for the total jump measure]
\label{TreeRep}
Assume that the tiling levels
 $p,q$ are two positive coprime integers. Then
\begin{equation}
\label{MesJInEq}
m(J) \leq \min \Big\{\frac{\al}{q}, \frac{\be}{p} \Big\}.
\end{equation}
\end{lem}

\begin{remark}
\label{RemLemTreeRep}
Let us explain our intuition behind \lemref{TreeRep}. Recall
that each vertex $w \in J$ is a representative of a connected 
component of $G(\Om)$ which is a finite  tree
(\lemref{JRep}). Let $H$ be  one of these connected 
components, and suppose that $H$ has
$m$ vertices in $A$ and $n$ vertices in $B$. Using
\lemref{DoubleCount}  it follows  that $m   p = n   q$. But
since  $p,q$ are now assumed to be positive coprime integers, this implies
that $q$ must divide $m$, and $p$ must divide $n$.
In particular, we  have
$m \geq q$ and $n \geq p$. Hence the connected component of
each vertex $w \in J$ contributes at least $q$ vertices to $A$, and at least $p$ vertices to $B$.
So intuitively we may expect to have 
\begin{equation}
\label{AqM}
 m(A) \geq q \cdot m(J),
\end{equation}
and 
\begin{equation}
\label{BqM}
 m(B) \geq p \cdot m(J).
\end{equation}
But notice that according to \eqref{eqFX4.1}
we have $m(A) = \al$, $m(B) = \be$, so that
the two inequalities \eqref{AqM} and \eqref{BqM}
together imply \eqref{MesJInEq}.
This explains why intuitively one may  expect that \lemref{TreeRep} should be true.
\end{remark}

We now turn to the formal proof of  \lemref{TreeRep}.
Let  $V =A \cup B$ be the vertex set 
of the graph $G(\Om) = (A,B,E)$,
and let $V' \sbt V$ be the set of those vertices whose connected component
in the graph $G(\Om)$ is a finite tree. Let
\begin{equation}
A' = V' \cap A, \quad B' = V' \cap B.
\end{equation}

\begin{lem}\label{mesAtag}
The sets $A', B'$ (and hence also $V' = A' \cup B'$) are measurable.
\end{lem}

\begin{proof}
Consider the sets   $J_A := J \cap A$ and $J_B := J \cap B$.
The sets $J_A$, $J_B$ are measurable since $A$, $B$ and $J$ are 
measurable sets.

Recall that every connected component in the graph $G(\Om)$ which 
is a finite tree, has a representative vertex $v \in J$ (\lemref{JRep}).
Hence $A'$ is the set of all vertices $a \in A$ such that
there is a finite path  connecting $a$ to some element $v \in J = J_A \cup J_B$.

We next observe that a vertex $a \in A$ is connected to some vertex $v \in J_A$
if and only if $a$ belongs to the set
$(\pal \circ \pbe^{-1} \circ \pbe \circ \pal^{-1})^n(J_A)$
for some $n$. This is because when moving from a vertex in $A$ to a neighbor vertex in $B$,
we first go from the  vertex to some of its incident edges (which corresponds to picking an element of $\Om$ belonging to the preimage of the vertex under the map $\pal$), and then go from this edge to its other endpoint vertex (which corresponds to taking the image of the edge under $\pbe$).
Similarly, when moving from a vertex in $B$ to a neighbor vertex in $A$, we first 
pick an edge in the preimage under the map $\pbe$ and then 
take the image of the edge under $\pal$.

For a similar reason,  a vertex $a \in A$ is connected to some vertex $v \in J_B$
if and only if $a$ is in the set
$ (\pal \circ \pbe^{-1} \circ \pbe \circ \pal^{-1})^n(\pal \circ \pbe^{-1})(J_B) $
for some $n$. 

(We note that here we consider $\pal$ and $\pbe$ as maps defined on the set $\Om$, 
 thus inverse images under these maps are understood to be subsets of $\Om$.)

We have thus shown that
\begin{equation}
\label{eqAP3.1}
A' = \bigcup_{n=0}^\infty \Big[
(\pal \circ \pbe^{-1} \circ \pbe \circ \pal^{-1})^n(J_A)  \cup 
 (\pal \circ \pbe^{-1} \circ \pbe \circ \pal^{-1})^n(\pal \circ \pbe^{-1})(J_B) \Big],
\end{equation}
so the measurability of $A'$ follows from the measurability of the sets $J_A, J_B$
and the fact that the measurability of a set is preserved under both images and preimages 
with respect to the maps $\pal$ and $\pbe$.

The proof that the set $B'$ is also measurable is  similar.
\end{proof}

\begin{proof}[Proof of \lemref{TreeRep}]
Assume that the tiling levels
 $p$ and $q$ are two positive coprime integers. 
We must prove that \eqref{MesJInEq} holds, or equivalently, 
that \eqref{AqM} and \eqref{BqM} are both satisfied.
We will prove \eqref{AqM} only. The proof of \eqref{BqM} is similar.

Recall that the connected component of any  vertex $v \in J$  in
 the  graph $G(\Om)$  is a finite tree (\lemref{JRep}). For each $v \in J$ we
let $h_A(v)$ be  the number of vertices of the
 connected component of $v$ which lie in the set $A$.
 Then $h_A$ is a nonnegative, integer-valued function on $J$. 
 Since each finite connected component
of  $G(\Om)$ must contain at least $q$ vertices 
 in $A$ (\remref{RemLemTreeRep}), we have
$h_A(v) \geq q$ for every $v \in J$.

We will show that the function $h_A$ is measurable and satisfies
\begin{equation}
\label{ineqHA}
\int_J h_A \leq m(A').
\end{equation}
Notice that once
 \eqref{ineqHA} is established, we can conclude (using $A' \sbt A$) that
\begin{equation}
\label{ineqAqJ}
m(A) \geq m(A') \geq \int_J h_A \geq q \cdot m(J),
\end{equation}
and \eqref{AqM} follows.
So it remains to show that $h_A$ is measurable and satisfies
\eqref{ineqHA}.

Recall that we denote by 
${L}^{(n)}$ the set of leaves removed at the $n$'th step
of the iterative leaves removal process, that is, if $n$ is even
then ${L}^{(n)} = A^{(n)}_1$ (the set of leaves in $A^{(n)}$), 
and if $n$ is odd then ${L}^{(n)} = B^{(n)}_1$ (the set of leaves in $B^{(n)}$). 
We construct  by induction two sequences $\{\phi_A^{(n)}\}$ and $\{\psi_A^{(n)}\}$ of functions on
 $V' = A' \cup B'$, as follows. We define
\begin{equation}
\label{eqdefpsi}
\psi_A^{(0)} = \1_{A'}, \quad 
\psi_A^{(n+1)} = \psi_A^{(n)} + \phi_A^{(n)},
\end{equation}
and
\begin{numcases}{\phi_A^{(n)}(v) =}
 - \psi_A^{(n)}(v), & $v \in V'  \cap L^{(n)}$ \label{phicase1} \\[6pt]
\sum_{w}  \psi_A^{(n)}(w), & $v  \in V' \setminus L^{(n)}$ \label{phicase2}
\end{numcases}
where  $w$ goes through all the vertices in $V'  \cap L^{(n)}$
who are neighbors of $v$ in the graph $G^{(n)}(\Om)$
(if there are no such neighbors then the sum is understood to be zero).

The motivation behind this construction is that
we view the function $\psi_A^{(n)}$ as assigning a certain mass to each vertex in $V'$.
We start with the function $\psi_A^{(0)}$ which  assigns a unit mass to each vertex in $A'$, and zero mass to each vertex in $B'$.  At the $n$'th step of the leaves removal process, 
by adding $\phi_A^{(n)}$ to $\psi_A^{(n)}$ we subtract the mass from each removed leaf in $V'$
and add the mass back to the neighbor of the leaf, so
that at each step the mass is transferred from the removed leaves to their neighbors.

In particular, each $\psi_A^{(n)}$ is a nonnegative, integer-valued function on $V'$.

Notice that whenever the leaves removal process exhausts
a connected component $H$ of the graph $G(\Om)$
 (so that $H$ is a finite tree by \lemref{lemJump}), then
the total mass accumulated
from all the $A'$-vertices of $H$ is transferred into 
the unique representative vertex $v$ of $H$ in the total jump set  $J$
(\lemref{JRep}), and the value of $\psi_A^{(n)}(v)$
will remain fixed from that point and on.
This implies that the sequence $\psi_A^{(n)}$ converges pointwise to
the function $h_A$ on $J$, and to zero on $V' \setminus J$.

We now give an equivalent way of constructing the function
$\phi_A^{(n)}$. 

Let $g_n$ be  a function on the set $E$ of 
edges of the graph  $G(\Om) = (A,B,E)$, defined as follows. We let $g_n(e) = \psi_A^{(n)}(v)$ if
$e$ is a leaf edge in the graph $G^{(n)}(\Om)$, incident  
to a vertex  $v \in V' \cap L^{(n)}$. If this is not the case, then we let $g_n(e)=0$.
Then $g_n$ is supported on the subset $E^{(n)}(V' \cap L^{(n)})$ of the
edge set $E$. By identifying the edges of $G(\Om)$ 
with elements of $\Omega$ (\remref{omegaEdgesId})  we may view
$g_n$ as a function on $\Om$.
 If $n$ is even, then
$g_n(x) = \psi_A^{(n)}(\pal(x))$ if $x$ is in the set
$ \pal^{-1}(V' \cap L^{(n)}) \cap \Om^{(n)}$,
and $g_n(x)=0$ otherwise.
Similarly, if $n$ is odd  then
$g_n(x) = \psi_A^{(n)}(\pbe(x))$ if $x$ is in the set
$ \pbe^{-1}(V' \cap L^{(n)}) \cap \Om^{(n)}$,
and $g_n(x)=0$ otherwise. 

Now consider again the definition \eqref{phicase1}, \eqref{phicase2}
 of the function $\phi_A^{(n)}$. Notice that up to a sign, both of 
 \eqref{phicase1} and \eqref{phicase2} involve summation of the 
 values $g_n(e)$ over all the edges $e$ incident to the vertex $v$, 
 with the only difference that in \eqref{phicase1} there is just one 
 such edge $e$ (since $v$ is a leaf in the graph $G^{(n)}(\Om)$)  while in \eqref{phicase2}
 the vertex $v$ might have several neighbors. Observe also that the identification 
 of edges in the graph with elements of $\Om$,  enables us 
 to express summation over neighbors in terms of the projections
  $\pal$ and $\pbe$. Thus if we extend the function $g_n$ to the whole 
  $\R$ by setting $g_n(x)=0$ for $x \notin \Om$, 
then the function $\phi_A^{(n)}$ can be equivalently defined by
$\phi_A^{(n)} = (-1)^{n+1} \pal(g_n)$ on $A'$, and
 $\phi_A^{(n)} = (-1)^n \pbe(g_n)$ on $B'$.

It follows from this equivalent definition, using induction on $n$, 
that $\phi_A^{(n)}$ and $\psi_A^{(n)}$ are both measurable functions on $V'$.
Moreover, these functions are both in $L^1(V')$, since we have
$\int_{\Om} g_n = \int_{V' \cap L^{(n)}} \psi_A^{(n)}$
and the projections $\pal$ and $\pbe$ map
$L^1(\Om)$ into $L^1(A)$ and $L^1(B)$ respectively.

As a consequence, the function $h_A$ (the pointwise limit of
$\psi_A^{(n)}$ on $J$) is a measurable function on $J$.

Now, notice that we have $\int_{V'}\psi_A^{(0)} = m(A')$. We claim that in fact 
$\int_{V'}\psi_A^{(n)} = m(A')$ for every $n$, that is, when the mass is transferred from
leaves to neighbors along the leaves removal process, 
the total mass remains constant. This is equivalent to the assertion that
$\int_{V'} \phi_A^{(n)} = 0$ for every $n$. Indeed,  we have
$\phi_A^{(n)} = (-1)^{n+1} \pal(g_n)$ on $A'$, and $\phi_A^{(n)} = (-1)^n \pbe(g_n)$ on $B'$.
Hence
\begin{equation}\label{massConst}
(-1)^{n} \int_{V'} \phi_A^{(n)} =  - \int_{A'} \pal(g_n) + \int_{B'} \pbe(g_n)
=  - \int_{\Om} g_n +  \int_{\Om} g_n = 0.
\end{equation}

We have thus proved that $\int_{V'} \psi_A^{(n)} = m(A')$ for every $n$. 
Moreover, the functions $\psi_A^{(n)}$ are nonnegative and converge
pointwise to  $h_A$ on $J$, and to zero on $V' \setminus J$.
Hence we may apply  Fatou's lemma, which yields
\begin{equation}
\label{eqFatou}
m(A') = \lim_{n \rightarrow \infty} \int_{V'} \psi_A^{(n)} \geq \int_{V'}
 \lim_{n \rightarrow \infty} \psi_A^{(n)} = \int_{J} h_A,
\end{equation}
and we arrive at \eqref{ineqHA}. Together with
\eqref{ineqAqJ}, this implies \eqref{AqM}.
The proof of \eqref{BqM} can be done in a similar way.
\lemref{TreeRep} is thus proved.
\end{proof}

\subsection{Proof of \thmref{thmA2.4}}
We can now turn to prove \thmref{thmA2.4}. Recall that the theorem asserts that 
if the tiling levels $p,q$ are two positive, coprime integers then
we must have 
\begin{equation}
\label{eqLW5.2}
m(\Om) > \al + \be - \min \Big\{\frac{\al}{q}, \frac{\be}{p} \Big\}.
\end{equation}

Let us combine the lower bound (\lemref{lemJMlower})
and the upper bound (\lemref{TreeRep}) 
for the total jump measure. Since by \eqref{eqFX4.1} 
we have $m(A) = \al$, $m(B) = \be$, this yields
\begin{equation}
\label{eqH4.4}
\al + \be - m(\Om) \leq m(J) \leq 
\min \Big\{\frac{\al}{q}, \frac{\be}{p} \Big\},
\end{equation}
and as a consequence,
\begin{equation}
\label{eqH4.5}
 m(\Om) \geq
  \al + \be -  \min\left\{\frac{\al}{q}, \frac{\be}{p}\right\}.
\end{equation}
We thus almost arrive at \eqref{eqLW5.2}. It only remains to show that 
an equality cannot  occur. 

We will need the following lemma.

\begin{lem}
\label{Jnull}
Suppose that we are given 
a full measure subset $\mathcal J$ of the total jump set $J$.
Let $\mathcal V$ be the set of all vertices $v \in V$
whose connected component in the graph
$G(\Om)$  intersects $\mathcal J$. Then $\mathcal V$ is a full
measure subset of $V'$.
\end{lem}

\begin{proof}
By definition,  $\mathcal V$ is the set of those vertices $v \in V$
such that there is a finite path  connecting $v$ to some
vertex in $\mathcal J$.   Let
\begin{equation}
\mathcal{J}_A = \mathcal J \cap A, \quad 
\mathcal{J}_B = \mathcal J \cap B, \quad
\mathcal{V}_A = \mathcal V \cap A, \quad 
\mathcal{V}_B = \mathcal V \cap B.
\end{equation}
By the same argument as in 
the proof of \lemref{mesAtag}, we have
\begin{equation}
\label{eqVA2.1}
\mathcal {V}_A  = \bigcup_{n=0}^\infty \Big[
(\pal \circ \pbe^{-1} \circ \pbe \circ \pal^{-1})^n(\mathcal {J}_A )   \cup 
 (\pal \circ \pbe^{-1} \circ \pbe \circ \pal^{-1})^n(\pal \circ \pbe^{-1})(\mathcal {J}_B) \Big].
\end{equation}
But the right hand sides of \eqref{eqAP3.1} 
and \eqref{eqVA2.1} coincide up to a set of measure zero,
since both the image and the preimage of a null set under 
the maps $\pal$ and $\pbe$  is again a null set. Hence
$\mathcal {V}_A$ is a full measure subset of $A'$.
In a similar way, 
$\mathcal {V}_B$ is a full measure subset of $B'$.
Thus $\mathcal {V} = \mathcal {V}_A \cup
\mathcal {V}_B$ is a full measure subset of $V' = A' \cup B'$.
\end{proof}

Now suppose to the contrary 
that there is equality in \eqref{eqH4.5}.
Then the two inequalities in \eqref{eqH4.4} also become equalities.
In particular,  we obtain
 \begin{equation}
\label{eqH4.6}
m(J) =  \min\left\{\frac{\al}{q}, \frac{\be}{p}\right\}.
\end{equation}
This means that we must have
$m(A) = q \cdot m(J)$ or $m(B) = p \cdot m(J)$. 
We shall suppose that  $m(A) = q \cdot m(J)$   (the other case is similar).
In turn, this implies that
all the inequalities in \eqref{ineqAqJ} become equalities, that is,
\begin{equation}
\label{eqAqJ}
m(A) = m(A') = \int_J h_A = q \cdot m(J).
\end{equation}

Since we know that 
$h_A(v) \geq q$ for every $v \in J$, it follows from
$\int_J h_A = q \cdot m(J)$ that $h_A(v) = q$ for every
$v$ in some full measure subset $\mathcal J$ of $J$.
 In other words, the connected component of  
every $v \in \mathcal J$ has  exactly $q$  vertices in $A$.
In turn, using \lemref{DoubleCount}  it follows that
any such a connected component must have
exactly $p$  vertices in $B$. For each $v \in J$,
let $h_B(v)$ be the number of vertices of the
 connected component of $v$ which lie in $B$.
  Then we have 
 $h_B(v) = p$ for every $v \in \mathcal J$,
 and hence a.e.\ on $J$.

Consider functions  $\phi_B^{(n)}$, $\psi_B^{(n)}$ on the
  set  $V' = A' \cup B'$, defined analogously to the functions
 $\phi_A^{(n)}$, $\psi_A^{(n)}$ from the
 proof of \lemref{TreeRep}. 
  Then the sequence $\psi_B^{(n)}$ converges pointwise to
$h_B$ on $J$, and to zero on $V' \setminus J$. 
But since the connected component of  
every $v \in \mathcal J$ has  exactly $p+q$ vertices 
($q$ of them in $A$, and $p$ of them in $B$) then 
after at most $p+q$ steps of the iterative leaves removal process,
all the mass $\psi_B^{(n)}$ on such a connected component
is concentrated on the representative vertex of the
connected component in $\mathcal J$. Using \lemref{Jnull} we know
that if $\mathcal V$ is the set of all vertices 
whose connected component 
intersects $\mathcal J$, then $\mathcal V$ is a full
measure subset of $V'$.  We conclude that
for all $n \geq p+q$ we have  $\psi_B^{(n)} = p$
on $\mathcal J = \mathcal V \cap J$
 (and hence, a.e.\ on $J$),
and $\psi_B^{(n)} = 0$
on $\mathcal V \setminus J$
 (and hence, a.e.\ on $V' \setminus J$). 
In turn, this implies that 
 $\int_{V'} \psi_B^{(n)} = 
 p \cdot m(J)$
for all $n \geq p+q$.
But on the other hand, observe (as in the
 proof of \lemref{TreeRep})  that we have
 $\int_{V'} \psi_B^{(n)} = m(B')$ for every $n$. 
We conclude that 
\begin{equation}
\label{massBP2}
m(B') = p \cdot m(J).
\end{equation}

Next we claim that $m(B') = m(B)$. Indeed, if not, then
$B \setminus B'$ is a positive measure subset of $B$. 
It follows that the set $\pal(\pbe^{-1}(B \setminus B') \cap \Om)$,
consisting of those vertices in $A$ that have a neighbor 
belonging to $B \setminus B'$ in the graph $G(\Om)$,
is a positive measure subset of  $A \setminus A'$.
But this implies that $m(A') < m(A)$, contradicting
\eqref{eqAqJ}. Hence we must have $m(B') = m(B)$.

We conclude that 
\begin{equation}
\label{massBP3}
m(B) = m(B') = p \cdot m(J).
\end{equation}

Finally we combine \eqref{eqAqJ}
and \eqref{massBP3} to obtain
\begin{equation}
\label{J5.7}
\frac{\alpha}{q} =  m(J) = \frac{\be}{p},
\end{equation}
and it follows that $p \al = q \be$. But this 
contradicts the assumption that $\al, \be$ are rationally independent. 
This establishes that equality in \eqref{eqH4.5} cannot occur,
and completes  the proof of \thmref{thmA2.4}.
\qed


\section{Simultaneous tiling by integer translates}
\label{secE9}

In this section 
 we turn  to deal with the case where the numbers
$\al, \be$ are linearly dependent over the rationals.
By rescaling, it would be enough to consider the
case $(\alpha,\beta) = (n,m)$ where $n,m$ are
positive integers.

We will prove Theorems 
\ref{thmA4.1},
\ref{thmA4.3} and
\ref{thmA4.2}
by showing that if 
a measurable function $f$ on $\R$ 
satisfies the simultaneous tiling condition
 \eqref{eqA2.7.1}
then  the tiling level vector 
$(p,q)$
must be proportional to $(m,n)$,
and if the level vector $(p,q)$ is nonzero then  the least possible measure of 
the support of $f$ is $n+m-\gcd(n,m)$.

The approach is based on a reduction of the 
simultaneous tiling problem from the real line
$\R$ to  the 
set of integers $\Z$.
In particular we will prove that
$n+m-\gcd(n,m)$
is also the least possible size of the support of a function
$g$ on $\Z$
that tiles the integers simultaneously (with a nonzero level vector) by 
two arithmetic progressions $\nz$ and $\mz$.

\subsection{} 
We begin by introducing the notion of tiling by translates of a function
on the set of integers $\Z$.
Let $g$ be a function on $\Z$, and $\Lam$ be a subset of $\Z$.
We say that $g+\Lam$ is a tiling of $\Z$ at level $w$ if we have
\begin{equation}
\label{eqH1.6}
\sum_{\lam \in \Lambda}g(t-\lam)=w,\quad t\in \Z,
\end{equation}
and the series \eqref{eqH1.6}
 converges absolutely for every $t \in \Z$.

We are interested in simultaneous tiling of the
integers by 
two arithmetic prog\-ressions
$\nz$ and $\mz$.
We thus  consider  
a function $g$ on $\Z$ satisfying 
\begin{equation}
\label{eqH2.7.1}
\sum_{k \in \Z}g(t-kn)=p,
\quad
\sum_{k \in \Z}g(t-km)=q,
\quad
t\in \Z,
\end{equation}
where  $n,m$ are positive 
integers, $p,q$ are complex numbers,
and both series in \eqref{eqH2.7.1} converge absolutely for every $t \in \Z$.

\subsubsection{}
We begin with the following basic result.

\begin{prop}
  \label{propH2.1}
Let $g$ be a 
 function on $\Z$ satisfying
\eqref{eqH2.7.1}, where
$n,m$ are positive integers. Then 
$g \in \ell^1(\Z)$, and the vector
$(p,q)$ must be proportional to $(m,n)$.
\end{prop}

\begin{proof}
First we observe that
\begin{equation}
\label{eqH2.7.8}
\sum_{t \in \Z} |g(t)| = 
\sum_{t=0}^{n-1} 
\sum_{k \in \Z} |g(t-kn)|.
\end{equation}
By assumption,  the inner sum on the right hand side
of \eqref{eqH2.7.8} converges for every $t$.
Hence the sum on the left hand side
converges as well, which shows that the function
$g$ must be in $\ell^1(\Z)$.
Next, we have
\begin{equation}
\label{eqH2.7.9}
\sum_{t \in \Z} g(t) = 
\sum_{t=0}^{n-1} 
\sum_{k \in \Z} g(t-kn) = np,
\end{equation}
where the last equality follows from 
condition \eqref{eqH2.7.1}.
In a similar way, we also have
\begin{equation}
\label{eqH2.7.10}
\sum_{t \in \Z} g(t) = 
\sum_{t=0}^{m-1} 
\sum_{k \in \Z} g(t-km) = mq,
\end{equation}
again using \eqref{eqH2.7.1}.
Hence $np=mq$, that is,  the vector
$(p,q)$ is proportional to $(m,n)$.
\end{proof}

\subsubsection{}
Let us recall (see  \defref{defDBA}) that 
an $n \times m$ matrix
$M = (c_{ij})$ is called a \emph{doubly stochastic array} 
if its entries $c_{ij}$ are
 nonnegative, and  the sum of the entries at each row is $m$
and at each   column  is $n$.
We have seen that
the minimal size of the
support of an $n \times m$
doubly stochastic array is
 $n + m - \gcd(n,m)$
(\thmref{thmC1.6}).
In the proof of  \lemref{lemFN1.2} 
we used one part of this result, namely, 
the part which
states that there exists an $n \times m$
doubly stochastic array whose support size 
is as small as  $n + m - \gcd(n,m)$.

In what follows 
we will use the other part of the result, that is,
the part which states 
 that $n + m - \gcd(n,m)$ constitutes a lower bound
for the support size of  any $n \times m$
doubly stochastic array. Actually, we will need a 
 stronger version of this result, proved in  \cite{EL22}, 
 which establishes that the same
lower bound holds also for complex-valued 
matrices, that is, even without assuming
that the matrix entries are nonnegative.

\begin{thm}[{see \cite[Theorem 3.1]{EL22}}]
\label{thmD1.1}
Let $M = (c_{ij})$ be an  $n \times m$ 
complex-valued matrix  satisfying \eqref{eqDS1.1}
and \eqref{eqDS1.2}, that is,
the sum of the entries at each 
row is $m$
and at each column  is
$n$. Then 
the support of $M$ 
has size at least $ n + m - \gcd(n,m)$.
\end{thm}

\subsubsection{}
By the \emph{support}
of   a function $g$ on $\Z$ we mean the set
\begin{equation}
\label{eqH2.1}
\supp g = \{t \in \Z : g(t) \neq 0\}.
\end{equation}

In the next result 
we  use \thmref{thmD1.1} to
give a lower bound for
 the support size of any function $g$ on $\Z$ 
that tiles simultaneously  by 
the two arithmetic progressions
$\nz$ and $\mz$
with a nonzero tiling level vector  $(p,q)$.

\begin{thm}
  \label{thmH3.1}
Let $g$ be a 
 function on $\Z$ satisfying
\eqref{eqH2.7.1} where 
$n,m$ are positive  integers
and the vector $(p,q)$ is nonzero. Then 
$\supp g$ has size at least $ n + m - \gcd(n,m)$.
\end{thm}

\begin{proof}
By \propref{propH2.1} the 
 function $g$ is in $\ell^1(\Z)$,
and the tiling level vector
$(p,q)$ is proportional to $(m,n)$.
By multiplying the 
 function $g$ on an
appropriate scalar we may
suppose that
$(p,q) = (m,n)$.

We will first prove the result in the special
case where $n,m$ are \emph{coprime}.
Let $\Z_{nm}$ be the
additive group of residue classes modulo $nm$.
Define a function
\begin{equation}
\label{eqH3.3}
h(t) := \sum_{k \in \Z} g(t - knm), \quad t \in \Z.
\end{equation}
Then $h$ is periodic with period $nm$, so it may be
viewed as a function on
 $\Z_{nm}$.

Let  $H_k$  denote the
subgroup of $\Z_{nm}$ generated
by the element $k$. One
can verify  using \eqref{eqH2.7.1}
and \eqref{eqH3.3} that the  function $h$
tiles the group $\Z_{nm}$ by translations
along each one of 
the two subgroups $H_n$ and $H_m$, that is to say,
\begin{equation}
\label{eqM1.15}
\sum_{s \in H_n} h(t-s) = m, \quad
\sum_{s \in H_m} h(t-s) = n, \quad
t \in \Z_{nm}.
\end{equation}

Next, we denote by $\Z_n$ and $\Z_m$ the
additive groups of residue classes modulo $n$
and $m$ respectively. 
Since $n,m$ are coprime, then by the
Chinese remainder theorem
 there 
is a group isomorphism 
$\varphi: \Z_{nm} \to \Z_n \times \Z_m$
given by
$\varphi(t) =  
(t \,\operatorname{mod} \, n, \,
t \,\operatorname{mod} \, m)$.
This isomorphism allows us to lift the
function $h$ to a new function 
\begin{equation}
\label{eqM3.2}
M: \Z_n \times \Z_m \to \R
\end{equation}
defined by 
$M(\varphi(t)) = h(t)$,
$t \in \Z_{nm}$. We use
\eqref{eqM3.2}
as an alternative way to
represent a complex-valued
 $n \times m$ matrix $M$,
in which  the rows of the matrix   are 
indexed by residue classes modulo $n$,
while  the columns 
indexed by 
residue classes modulo $m$.

We now claim that the sum of the entries 
of the matrix $M$ at each row is equal to
$m$ and at each column is equal to $n$.
To see this, we observe that 
 the isomorphism 
$\varphi$ maps the subgroup
$H_n$ of $\Z_{nm}$ onto 
the subgroup
$\{0\} \times \Z_m$
of $\Z_n \times \Z_m$.
Hence for each $i \in \Z_n$, the set
$\{(i,j) : j \in \Z_m\}$ is the image under
$\varphi$ of a certain coset of $H_n$ in $\Z_{nm}$,
say, the coset $a_i - H_n$. It follows that 
\begin{equation}
\label{eqM1.12}
\sum_{j \in \Z_m} M(i,j) = 
\sum_{s \in H_n } h(a_i-s) = m,
\end{equation}
where in the last equality we used \eqref{eqM1.15}.
In a similar way, $\varphi$ maps the subgroup
$H_m$ onto  $\Z_n \times \{0\}$,
so for each $j \in \Z_m$ the set
$\{(i,j) : i \in \Z_n\}$ is the image under
$\varphi$ of a coset $b_j - H_m$, and we obtain
\begin{equation}
\label{eqM1.13}
\sum_{i \in \Z_n} M(i,j) = 
\sum_{s \in H_m } h(b_j-s) = n,
\end{equation}
again using \eqref{eqM1.15}.
We thus see that
the sum of the entries of $M$ at each row is
$m$ and at each column is $n$.

Notice that we cannot  say that $M$
is a doubly stochastic array, since 
the entries of $M$ are 
not guaranteed  to be  nonnegative
 (see  \defref{defDBA}).
Nevertheless, we can now invoke 
\thmref{thmD1.1} which is valid also
for complex-valued 
matrices. Since $n,m$ are coprime, it
 follows from the theorem that
the support of $M$ has size at least
$n + m-1$. Since $\supp h$ and
$\supp M$ are of the same size,
we conclude that
\begin{equation}
\label{eqM1.20}
|\supp h| \geq n+m-1.
\end{equation}

Lastly we observe that  if $h(t) \neq 0$ for some $t \in \Z$,
then 
$g$ does not vanish on at least one element
of the arithmetic progression $\{t - knm :  k \in \Z\}$
due to  
\eqref{eqH3.3}.
But these  arithmetic progressions are pairwise disjoint
as $t$ goes through 
 a complete set of residues modulo $nm$.
This shows that $\supp g$ has size
at least as large as the size of $\supp h$.
So combined with \eqref{eqM1.20}
this implies that
$\supp g$ is of size at least $n+m-1$.

We have thus proved the result in the special
case where $n,m$ are coprime. To prove the result
in the general case, we now let
$n,m$ be two arbitrary positive integers.
We then write $n = dn'$, $m=dm'$, where $d = \gcd(n,m)$
and $n',m'$ are coprime.
For each $0 \leq j \leq d-1$ we consider the function
\begin{equation}
\label{eqM1.25}
g_j(t) := g(j + dt), \quad t \in \Z.
\end{equation}
It follows from \eqref{eqH2.7.1} 
that each $g_j$ tiles the integers simultaneously by
the two arithmetic progression $n' \Z$ and $m' \Z$
at levels $p$ and $q$ respectively. Since $n',m'$
are coprime (and the tiling level vector is nonzero)
then, by what we have proved above, the size of
$\supp g_j$ must be at least $n'+m'-1$. It follows that
\begin{equation}
\label{eqM1.27}
| \supp g | = \sum_{j=0}^{d-1} |\supp g_j| \geq
d(n'+m'-1) = n+m-d,
\end{equation}
and we arrive at the desired conclusion.
\end{proof}

 We note that the correspondence between  the
$n \times m$ doubly 
stochastic arrays and the nonnegative
functions which 
tile the group $\Z_{nm}$ 
by translations
along each one of 
the two subgroups $H_n$ and $H_m$
(where $n,m$ are coprime)
was pointed out in an earlier version of
\cite{KP22}.

\subsubsection{}
Our next result shows that
the lower bound in \thmref{thmH3.1} is in fact sharp.

\begin{thm}
  \label{thmH3.2}
For any two positive  integers $n,m$ 
there exists a nonnegative function $g$ on $\Z$,
supported on a set of $n+m-\gcd(n,m)$ consecutive
integers,  and satisfying
\eqref{eqH2.7.1} with $(p,q)=(m,n)$.
\end{thm}

\begin{proof}
Let $\chi_k$  denote the indicator
function of the subset
$\{0,1,\dots,k-1\}$ of $\Z$.
We consider 
 a function $g$ on $\Z$ defined
as the convolution
\begin{equation}
\label{eqM3.4}
g(t) = (\chi_n \ast \chi_m)(t) = \sum_{s \in \Z} \chi_n(t-s)\chi_m(s),
\quad
t \in \Z.
\end{equation}
Then $g$ is supported on the set
$\{0,1,\dots, n+m-2\}$ of size $n+m-1$.
Since the function $\chi_n$ tiles  at level one
by translation with $\nz$, 
and $\chi_m$ tiles also at level one
by translation with $\mz$, we can deduce from
\eqref{eqM3.4} that
$g$   satisfies the simultaneous tiling condition
\eqref{eqH2.7.1} with $(p,q)=(m,n)$.
This proves the result in the 
case where $n,m$ are coprime.

 To prove the result
in the general case, we  
write as before $n = dn'$, $m=dm'$, where $d = \gcd(n,m)$
and $n',m'$ are coprime.
Let $h$ be a nonnegative function  on $\Z$,
supported on a set of $n'+m'-1$ consecutive
integers, which  tiles simultaneously by
the two arithmetic progression $n' \Z$ and $m' \Z$
at levels $m$ and $n$ respectively
(such a function $h$ exists, by what we have proved above).
We then define a  function  $g$ on $\Z$ by
\begin{equation}
\label{eqM1.34}
g(j+dt) := h(t), \quad 0 \leq j \leq d-1, \; t \in \Z.
\end{equation}
Then $g$ satisfies the simultaneous tiling
condition \eqref{eqH2.7.1}, and $g$ is supported on
a set of  $d(n'+m'-1) = n+m-d$
consecutive
integers,  as required.
\end{proof}

\subsection{}

Given  a measurable 
 function $f$ on $\R$, we define for each
$x \in \R$  a function $f_x$ on the set of integers $\Z$, given by
\begin{equation}
\label{eqH1.4}
f_x(t) := f(x+t), \quad t \in \Z.
\end{equation}

The following lemma gives a connection between tilings
of $\R$ and tilings of $\Z$.

\begin{lem}
  \label{lemH1.1}
Let $f$ be a measurable function on $\R$, 
and $\Lam$ be a subset of $\Z$. Then
$f+\Lam$ is a tiling of $\R$
at level $w$ if and only if
$f_x + \Lam$ is a tiling of $\Z$
at the same level $w$
for almost every $x \in \R$.
\end{lem}

\begin{proof}
Let
$f+\Lam$ be a tiling of $\R$
at level $w$, then we have
\begin{equation}
\label{eqH1.45}
\sum_{\lambda\in\Lambda}f(x-\lambda)=w
\end{equation}
for all $x$ in some set $E \sbt \R$ of full measure.
It follows that for $t \in \Z$ we have
\begin{equation}
\label{eqH1.46}
\sum_{\lambda\in\Lambda}f_x(t-\lambda)=
\sum_{\lambda\in\Lambda}f(x+t-\lambda)=w
\end{equation}
provided that $x \in E-t$.
(The series in both \eqref{eqH1.45} and 
\eqref{eqH1.46} are understood to converge absolutely.)
Hence $f_x + \Lam$ is a tiling of $\Z$
at level $w$
for every $x$ belonging to  the set 
$ \bigcap_{t \in \Z} (E - t)$, which is also
a set of full measure in $\R$.

Conversely, let
$f_x + \Lam$ be a tiling of $\Z$
at level $w$
for almost every $x \in \R$. Then 
\begin{equation}
\label{eqH1.48}
\sum_{\lambda\in\Lambda}f(x-\lambda)=
\sum_{\lambda\in\Lambda}f_x(-\lambda)=w\quad\text{a.e.}
\end{equation}
(with absolute convergence of the series) and so
$f+\Lam$ is a tiling of $\R$ at level $w$.
\end{proof}

\begin{lem}
  \label{lemH4.1}
Let $f$ be a measurable function on $\R$.
Then
\begin{equation}
\label{eqH4.1}
\mes (\supp f) = \int_0^1 |\supp f_x|\, dx.
\end{equation}
\end{lem}

The proof of this lemma is standard and so we omit the details.

\subsection{}
Now we can prove
\thmref{thmA4.1},  \thmref{thmA4.3} and \thmref{thmA4.2}.

\begin{proof}[Proof of \thmref{thmA4.1}]
Let $n,m$ be positive  integers, and
 $f$ be a measurable function on $\R$ 
satisfying \eqref{eqA2.7.1}. 
By \lemref{lemH1.1}, the function 
$g = f_x$ then satisfies the simultaneous tiling condition
\eqref{eqH2.7.1} for almost every
$x \in \R$. Using \propref{propH2.1}
we conclude that the vector
$(p,q)$ must be proportional to $(m,n)$.
\end{proof}

\begin{proof}[Proof of \thmref{thmA4.3}]
Let $f$ be a measurable function on $\R$ satisfying
\eqref{eqA2.7.1} where 
$n,m$ are positive integers
and the vector $(p,q)$ is nonzero.
By \lemref{lemH1.1}, the function 
$g = f_x$ then satisfies the 
simultaneous tiling condition 
\eqref{eqH2.7.1} for almost every
$x \in \R$. 
By applying \thmref{thmH3.1}
to the function $g = f_x$ we obtain
that $|\supp f_x| \geq n+m-\gcd(n,m)$
for  almost every
$x \in \R$. Finally, combining this with
\lemref{lemH4.1}
we conclude that
\begin{equation}
\label{eqH4.63}
\mes (\supp f) = \int_0^1 |\supp f_x|\, dx
\geq n+m-\gcd(n,m),
\end{equation}
and so the theorem is proved.
\end{proof}

\begin{proof}[Proof of \thmref{thmA4.2}]
Let $n,m$ be positive  integers, and 
$(p,q)=(m,n)$. Let $g$ be the function given
by \thmref{thmH3.2}, that is, $g$ is a
 nonnegative function on $\Z$,
supported on a set of $n+m-\gcd(n,m)$ consecutive
integers,  and satisfying
\eqref{eqH2.7.1}. We then construct a measurable
(in fact, piecewise constant) nonnegative
 function $f$ on $\R$ given by
$f(x+t) = g(t)$ for every $t \in \Z$ and $x \in [0,1)$.
Then  $f$ is supported on an
interval of length $n+m-\gcd(n,m)$, and $f$
   satisfies 
the tiling condition
\eqref{eqA2.7.1}  by \lemref{lemH1.1}.
\end{proof}


\section*{Acknowledgement}
We thank  Mihalis Kolountzakis for posing to us
the problem discussed in \secref{secY2}.


\end{document}